\documentclass[11pt]{article}
\usepackage[ansinew]{inputenc}
\usepackage{graphicx}
\usepackage{color}
\usepackage{enumerate,latexsym}
\usepackage{latexsym}
\usepackage{amsmath,amssymb}
\usepackage{graphicx}
\usepackage[margin=1.4in]{geometry}
\usepackage{amsthm}
\newfont{\bb}{msbm10 at 11pt}
\newfont{\bbsmall}{msbm8 at 8pt}

\def\R{\mathbb{R}}

\def\r{\mathbb{R}}

\def\N{\mathbb{N}}

\def\esf{\mathbb{S}}

\newcommand{\nc}{\newcommand}
\newcommand{\ben}{\begin{enumerate}}
\newcommand{\bit}{\begin{itemize}}
\newcommand{\een}{\end{enumerate}}
\newcommand{\eit}{\end{itemize}}

\newcommand{\Ric}{\mbox{\rm Ric}}

\newcommand{\bp}{\begin{proof}}
\newcommand{\ep}{\end{proof}}

\DeclareMathOperator{\Ind}{Ind}

\DeclareMathOperator{\dist}{dist}
\DeclareMathOperator{\dvg}{div}

\DeclareMathOperator{\area}{area}
\DeclareMathOperator{\genus}{genus}

\DeclareMathOperator{\graph}{graph}

\def\r{{\rho}}
\def\s{{\sigma}}
\def\ov{\overline}

\def\de{{\delta}}

\def\ve{{\varepsilon}}

\newcommand{\pl}[2]{{\frac{\partial #1}{\partial #2}}}

\newcommand{\h}{\mathcal{H}}

\newcommand{\In}{\subset}
\newcommand{\Om}{\Omega}
\newcommand{\om}{\omega}
\newcommand{\dl}{{\delta}}
\newcommand{\Dl}{{\Delta}}

\newcommand{\al}{{\alpha}}

\newcommand{\ed}{{\rm d}}

\newcommand{\eps}{{\varepsilon}}
\newcommand{\fr}[2]{\frac{#1}{#2}}
\newcommand{\sm}{{\setminus}}

\newcommand{\vlinesub}[1]{\vline_{_{_{_{_{_{#1}}}}}}}
\newcommand{\la}{\langle}
\newcommand{\ra}{\rangle}

\newtheorem{theorem}{Theorem}[section]
\newtheorem{lemma}[theorem]{Lemma}
\newtheorem{proposition}[theorem]{Proposition}
\newtheorem{thm}[theorem]{Theorem}
\newtheorem{remark}[theorem]{Remark}

\newtheorem{corollary}[theorem]{Corollary}
\newtheorem{definition}[theorem]{Definition}

\newtheorem{claim}[theorem]{Claim}

\nc{\bl}{\begin{lemma} }

\nc{\el}{\end{lemma} }

\nc{\bt}{\begin{theorem} }

\nc{\et}{\end{theorem} }
\newcommand{\rc}{ \renewcommand }

\rc{\v}{    \overset{\longrightarrow} }

\definecolor{rr}{rgb}{.8,0,.3}

\newtheoremstyle{TheoremNum}
        {\topsep}{\topsep}              
        {\itshape}                      
        {}                              
        {\bfseries}                     
        {.}                             
        { }                             
        {\thmname{#1}\thmnote{ \bfseries #3}}
    \theoremstyle{TheoremNum}
    \newtheorem{thmn}{Theorem}

\begin{document}

\begin{title}
{CMC hypersurfaces with bounded Morse index}
\end{title}
\begin{author}
{Theodora Bourni\thanks{ The first author was supported by grant 707699 of the Simons Foundation} \and Ben Sharp \and Giuseppe
   Tinaglia
 }
\end{author}
\maketitle
\begin{abstract}
We develop a bubble-compactness theory for embedded CMC hypersurfaces with bounded index and area inside closed Riemannian manifolds in low dimensions. In particular we show that convergence always occurs with multiplicity one, which implies that the minimal blow-ups (bubbles) are all catenoids. We also provide bounds on the area of separating CMC surfaces of bounded (Morse) index and use this, together with the previous results, to bound  their genus.

\par
\vspace{.1cm} \noindent{\it Mathematics Subject Classification:}
Primary 53A10, Secondary 49Q05, 53C42.

\vspace{.1cm} \noindent{\it Key words and phrases:} Minimal surface,
constant mean curvature, finite index, curvature estimates, area estimates.
\end{abstract}


\section{Introduction}

Throughout this paper, $N$ will be a closed (compact and without boundary) Riemannian $n$-manifold of dimension $n\le 7$ and an $H$-hypersurface $M\In N$ will be a closed connected hypersurface embedded in $N$ with constant mean curvature (CMC) $H>0$.

We prove compactness results for sequences of $H$-hypersurfaces in $N$: this study is inspired by the result by Choi and Schoen~\cite{cs1} that
the moduli space of fixed genus closed minimal surfaces embedded in $( \esf^3, h )$ with a metric $h$ of positive Ricci curvature
has the structure of  a compact real analytic variety, see Theorem~\ref{thm:conv}. 

 \begin{theorem}\label{thm:conv2} Let $3\leq n \leq 7$. Given $H>0$, let $\{M_k\}_{k\in \N}$ be a sequence of $H$-hypersurfaces in $N^n$ satisfying
 \[
 \sup_k \h^{n-1}(M_k)<\infty \quad \text{and}\quad \sup_k \Ind_0(M_k)<\infty.
 \]
 Then, there exists a hypersurface $M_\infty$ effectively embedded in $N$ with constant mean curvature $H$ and a finite set of points $\Delta\subset N$ such that, after passing to a subsequence, $\{M_k\}_{k\in \N}$ converges smoothly and with \textbf{multiplicity one}, to $M_\infty$ away from $\Delta$. Furthermore $\Dl$ is contained in the non-embedded part of $M_\infty$. 
 \end{theorem}

Contrary to the setting of minimal hypersurfaces, it is possible that a sequence of embedded $H$-hypersurfaces ($H>0$) converges to a limit which is itself not embedded. For instance a sequence of degenerating Delaunay surfaces converges to a string of pearls - CMC spheres which self-intersect tangentially. We refer to connected collections of $H$-hypersurfaces which meet tangentially as ``effectively embedded'' (see Definition \ref{def:ee}). Here $\Ind_0$ refers to the number of negative eigenvalues of the Jacobi operator when restricted to volume-preserving deformations (see Section \ref{prelim}).
Our compactness theorem guarantees that any weak-limit of a sequence of $H$-hypersurfaces with bounded Morse index ($\Ind_0$) and area is effectively embedded and obtained via multiplicity one graphical convergence away from finitely many points.

 \begin{remark}
Notice that the convergence always happens with multiplicity one as a result of the strict positivity of the mean curvature $H>0$ and not by an assumption on the ambient manifold. If $\{M_k\}$ are all separating and stable ($\Ind_0=0$), multiplicity one convergence has been obtained in \cite[Theorem 2.11 (ii)]{ZZ19} where in this case $\Dl = \emptyset$ by the regularity theory for stable CMC hypersurfaces (see e.g. Lemma \ref{sreg}). The theorem above shows that multiplicity one convergence continues to hold  under bounded index, regardless of whether $\Dl$ is empty or not.   These facts are in sharp contrast to the setting of minimal hypersurfaces where higher multiplicity convergence is guaranteed if $\Dl \neq \emptyset$, or ruled out altogether (for instance) under the assumption that $\Ric_N >0$ \cite{Ben}.

Furthermore when $n=3$, in view of Theorem  \ref{thm_area}, we can replace the assumed area bound with the topological condition that each $M_k$ separates $N$.   \end{remark}

In \cite{bw19} the authors develop an extensive regularity and compactness theory for codimension 1 integral varifolds with constant mean curvature and finite index in a Riemannian manifold of any dimension. They in fact deal with a much larger class of varifolds with appropriately bounded first variation.

Inspired by the bubbling analysis carried out in \cite{BS18}, and the works of Ros \cite{Ros95} and White \cite{W15}, we are able to capture shrinking regions of instability along a convergent sequence $M_k$ to provide a more refined picture close to $\Dl$. As in \cite{BS18} we can blow-up these regions to obtain complete embedded minimal hypersurfaces in $\R^n$ (``bubbles'') which themselves have finite index and Euclidean volume growth. A key feature in the setting of $H$-hypersurfaces ($H>0$), is that the multiplicity one convergence guaranteed by Theorem \ref{thm:conv2} implies that all bubbles have two ends (since they occur at the non-embedded part of the limit) and are therefore catenoids thanks to the classification results of Schoen \cite{sc1}. The full statements of these results can be found in Section \ref{sec_bubb}, but for now we content ourselves with stating the following corollary of the bubble-compactness Theorem \ref{thm:bubb}: 

\begin{theorem}[Corollary \ref{cor_bubb}]
Let $3\leq n \leq 7$ and $H>0$. 
Then there exists $\mathcal{G}=\mathcal{G}(N,\Lambda, \mathcal{I},H)$ so that the collection of $H$-hypersurfaces with index bounded by $\mathcal{I}$ and volume bounded by $\Lambda$ has at most $\mathcal{G}$ distinct diffeomorphism types. Furthermore for any $H$-hypersurface $M$ with the above index and volume bounds we have uniform control on the total curvature 
\[\int_M |A|^{n-1} \leq \mathcal{G}.\]
\end{theorem}

Finally, motivated  by the results in~\cite{MT18}, when $n=3$ we provide bounds on the area of separating CMC surfaces of bounded (Morse) index and use this, together with the previous results, to bound  their genus as well.

\begin{theorem} \label{main}
Given $\mathcal I\in \mathbb N$ and $H>0$, let $M$ be an $H$-surface in $N$ with index bounded by $\mathcal{I}$. If we furthermore assume that \textbf{either} 
\begin{enumerate}
\item $M$ is separating in $N$, \textbf{or}
\item $N$ has finite fundamental group (e.g. if $N$ has positive Ricci curvature) 
\end{enumerate}
 then there exists a constant $\mathcal{A}:=\mathcal{A}(\mathcal{\mathcal I}, H, N)$ such that  
\[
\genus(M)+\area(M)\leq \mathcal{A}.
\]
\end{theorem}

Since there exist examples of connected closed 
minimal surfaces embedded in and separating a flat 3-torus with arbitrarily large area but bounded index~\cite{tra5}, see Remark~\ref{example}, having $H>0$ is a necessary hypotheses to obtain an area estimate. For minimal surfaces embedded in closed three-manifolds $N$ with positive scalar curvature $R_N >0$, an analogous result has been obtained in \cite{chkm}. Indeed, Theorem~\ref{main} has been proven independently by Saturnino \cite{abs} using  techniques developed in \cite{chkm}. Finally, for arbitrary three-manifolds $N$ and \emph{immersed} CMC surfaces $\Sigma\In N$ with sufficiently large mean curvature $H_\Sigma > H_0$, an effective (and linear) genus bound in terms of index has been obtained in \cite{NH19}. 

We will prove the area estimate in Section \ref{sec_area} (See Theorem \ref{thm_area} and  Corollary \ref{cor_area}), the genus bound will then follow from a general bubble-compactness argument for $H$-hypersurfaces with bounded index and area, the full details of which appear in Section~\ref{sec_bubb}. 
\section{Preliminaries} \label{prelim}

Let $N^n$ be a closed (compact and without boundary) Riemannian $n$-manifold, where here and throughout we restrict $3\leq n \leq 7$. 
\begin{definition}
An $H$-hypersurface $M\In N$ will be a closed connected hypersurface embedded in $N$ with constant mean curvature $H>0$. When $n=3$ we will often refer to $M$ as an $H$-surface. 	
\end{definition}

Let  $\mu$ be the canonical measure corresponding to the metric on $M$ (inherited by the metric on $N$), $\nu$ a choice for its unit normal and $A$ the second fundamental form of the embedding.  We consider $Q$, the quadratic form associated to the Jacobi operator:
\[
Q(u,u)=\int_M|\nabla u|^2-(|A|^2+\Ric_N(\nu, \nu)) u^2 \,d\mu\,,\,\,u\in W^{1,2}(M)\,,
\]
where $\Ric_N$ is the Ricci curvature of $N$. 

Recall that for an open set $U\subset N$, the index of $M$ in $U$, $\Ind(M\cap U)$, is defined as the index of $Q$ over $W_0^{1,2}(M\cap U)$, that is, by the minimax classification of eigenvalues, the maximal dimension of the vector subspaces $E\subset \{u\in W_0^{1,2}(M\cap U): Q(u,u)< 0 \}$. 

Constant mean curvature (CMC) hypersurfaces are critical points of the area ($\h^{n-1}$-measure) functional for variations which preserve the signed volume ($\h^{n}$-measure). This can be characterised infinitesimally as all variations whose initial normal speed $u$ satisfies $\int_Mu\,d\mu=0$. Thus it makes sense to define a new index $\Ind_0(M\cap U)$ as the index of $Q$ over 
$$\dot W_0^{1,2}(M\cap U)=\{u\in W_0^{1,2}(M\cap U):\int_{M\cap U}u\,d\mu=0\}$$ that is the maximal  dimension of the vector subspaces $\tilde{E}\subset \{u\in \dot W_0^{1,2}(M\cap U): Q(u,u)< 0\}$. 
We will call the CMC surface $M$ \emph{stable} (in $U$) if $\Ind_0(M)=0$ ($\Ind_0(M\cap U)=0$) and \emph{strongly stable} (in $U$) if $\Ind(M)=0$ ($\Ind(M\cap U)=0$). 
Note that if $U\subset W \subset N$ are open sets, then $\Ind(M\cap W)\geq \Ind(M\cap U)$ and $\Ind_0(M\cap W)\geq \Ind_0(M\cap U)$ and the two indices satisfy the following relation.

\begin{lemma}\label{lem:index} For any $k\in \N\cup\{0\}$ we have
\[
\Ind_0(M)=k \implies k\le \Ind(M)\leq k+1\,.
\]
\end{lemma}
\begin{proof}It follows trivially from the definition of our indices that $\Ind_0(M) \leq \Ind(M)$. So suppose that the lemma is not true and instead we have $\Ind(M)\geq k+2$. Thus there exists a $k+2$-dimensional vector subspace $E\In W^{1,2}(M)$ with $Q(f,f)<0$ for all $f\in E$. Let $E^\top=\{f\in E : \int_M f = 0\} \In \dot W^{1,2}(M)$, then $\dim E^\top \geq k+1$ and we still have $Q(f,f)<0$ for all $f\in E^\top$ giving $\Ind_0 (M) \geq k+1$, a contradiction.  
\end{proof}

%
%

Next we remind the reader of the curvature estimates available for stable $H$-hypersurfaces via the work of Lopez--Ros \cite{lor2} when $n=3$ and Schoen-Simon \cite{SS81} when $n\geq 4$. 

\begin{lemma}\label{sreg}
Let $H>0$ be fixed and $M^{n-1}\In N^n$ an $H$-hypersurface. Given $p\in M$ and $\rho>0$, assume that $M\not\subset B^N_{\rho}(p)$ and that either
\begin{itemize}
\item[(i)] $n=3$, $\Ind_0(M\cap B_\r^N(p))=0$  or
\item[(ii)] $n\le 7$, $\Ind(M\cap B_\r^N(p))=0$ and $\rho^{-(n-1)}\mathcal H^{n-1}(M\cap B^N_\rho(p))\le \mu$\,.
\end{itemize}
 Then,
\[
 |A|(p)\leq \fr{C}{\rho}\,,
 \]
 where $C$ is a constant that depends on $N$, the value of the mean curvature and, in case (ii),  also on $\mu$ .
\end{lemma}

\begin{proof} 
The proof is by contradiction, so we suppose that we have a sequence of $H$-hypersurfaces $\{M_k\}_{k\in \N}$, $p_k\in M_k$ and $\rho_k>0$ such that  $M_k\not \subset B^N_{\rho_k}(p_k)$ and 
\[
\rho_k |A_k|(p_k)\geq k,
 \]
 where  $|A_k|$ is the norm of the second fundamental form of $M_k$.
Abusing the notation, let $M_k$ denote the connected component of $M_k\cap B^N_{\rho_k}(p_k)$ containing $p_k$ and let 
\[
a_k:=|A_k|(q_k)\dist_N(q_k,\partial B^N_{\rho_k}(p_k)) = \max_{q\in M_k} |A_k(q)|\dist_N(q,\partial B^N_{\rho_k}(p_k))\geq |A_k|(p_k)\rho_k \geq k.
\]
Using the notation $d_k=\dist_N(q_k,\partial B_{\rho_k}(p_k))$, we rescale $B^N_{d_k}(q_k)$ by $|A_k|(q_k)$ and denote by $\widetilde M_k$ the scaled connected component of $M_k \cap B^N_{d_k}(q_k)$ containing $q_k$, where the scaling is done in geodesic coordinates with origin at $q_k$. Note that $d_k$ is bounded and since $|A_k|(q_k)\to\infty$ and $a_k:=d_k|A_k|(q_k)\to \infty$,  then $\widetilde{M}_k$ is a sequence of  CMC hypersurfaces in $B_{a_k}(0)$ equipped with metrics $g_k$  which converge in $C^2$ to the Euclidean metric and whose mean curvature $\widetilde{H}_k = |A_k|(p_k)^{-1}H$ converges to $0$. 
Moreover, $|\widetilde{A}_k(0)|\equiv 1$ for all $k$ and for $z\in B_{\fr{a_k}{2}}(0)$ we have that $|\widetilde{A}_k(z)| \leq 2$. Furthermore $\Ind_0 (\widetilde{M}_k\cap B_{\fr{a_k}{2}}(0))=0$ when $n=3$ and $\Ind (\widetilde{M}_k\cap B_{\fr{a_k}{2}}(0))=0$ when $3<n\leq 7$. 

Thus, after passing to a subsequence,  $\widetilde{M}_k$ converges (locally uniformly) in $C^2$ to some complete minimal surface $\widetilde{M}_\infty$ embedded in $\R^{n}$ with $\Ind_0(\widetilde M_\infty)=0$ in case $n=3$ and $\Ind(M_\infty)=0$ in case $3<n\le 7$. For the case $n=3$, by Lopez--Ros \cite{lor2}, $M_\infty$  is a plane, contradicting  that $|A_\infty (0)| = 1$. In case $3<n\le 7$, $\widetilde{M}_\infty$ is a stable minimal surface which, by the monotonicity formula (applied to each $\widetilde M_k$) and the assumption on the $\h^{n-1}$-measure, has  Euclidean volume growth. Therefore, the curvature estimates of Schoen--Simon \cite{SS81} imply that $M_\infty$ must be a plane which contradicts that $|A_\infty (0)| = 1$.
\end{proof}

\begin{remark}\label{sregrem}
The estimates for the norm of the second fundamental form in (ii) of Lemma~\ref{sreg} also hold when $\Ind_0(M)=0$~\cite{bcw18}. The proof follows from the same scaling argument once the authors prove that the hyperplane is the only complete connected oriented stable minimal hypersurface embedded in $\mathbb{R}^n$ that has Euclidean area growth and no singularities. We note that in~\cite{bcw18} our notion of being stable with respect to volume preserving variations is referred to as weak stability. We also note that a key ingredient in proving this characterization of the hyperplane is the fact that a complete connected oriented stable minimal hypersurface immersed in $\mathbb{R}^n$  is one ended~\cite{chchzh}.
\end{remark}

   \begin{definition} \label{def:lbsf}  Let $U$ be an open set in $N$ and let $\{M_k\}_{k\in \N}$ be a sequence of $H$-hypersurfaces in $N$.
We say that the sequence $\{M_k\}_{k\in \N}$  has
{\em locally bounded norm of the second fundamental form in $U$} if for each
compact set $B$ in $U$, 
\[
\sup_k\sup_{M_k\cap B}|A_{M_k}|<\infty
\]
where $|A_{M_k}|$ is
the norm of the second fundamental form of  $M_k$.
\end{definition}

\begin{definition}\label{def:singset}
 Let $\{M_k\}_{k\in \N}$ be a sequence of $H$-hypersurfaces in $N$. A closed set $\Delta\subset N$ is called a {\em singular set of convergence} if, after passing to a subsequence and reindexing, we have the following.
 \begin{itemize}
 \item For any $q\in \Delta$, $\rho>0$ and $n\in\N$, $\sup_k \sup_{M_k\cap B^N_\rho(q)}|A_{M_k}|>n$;
 \item $\{M_k\}_{k\in \N}$  has 
locally bounded norm of the second fundamental form in $N\setminus\Delta$. 
\end{itemize}
A point $q\in \Delta$ will then be called a {\em singular point of convergence}.
\end{definition}

Note that $\Delta$, as in Definition \ref{def:singset}, is not uniquely defined. However, when  $\{M_k\}_{k\in \N}$ does not have
 locally bounded norm of the second fundamental form in $N$, we can always construct a singular set, for instance as follows. For each $k\in \N$, let the maximum of the norm of the second fundamental form $|A_{M_k}|$ of $M_k$ be achieved at a point $p_{1,k}\in M_k$.  After choosing a subsequence and reindexing, we obtain a sequence $M_{1,k}$
such that the points $p_{1,k}\in M_{1,k}$ converge to a point $q_1\in N$.  Suppose the sequence of hypersurfaces $M_{1,k}$ fails to have locally bounded norm of the second fundamental form in $N\setminus\{q_1\}$.
Let $q_2\in N\setminus\{q_1\}$ be a point that is furthest away from $q_1$ and such that,
after passing to a subsequence $M_{2,k}$,  there exists a sequence of points $p_{2,k}\in M_{2,k}$
converging to $q_2$ with $\lim_{k\to\infty}A_{M_{k,2}}(p_{2,k})=\infty$. If the sequence of
hypersurfaces $M_{2,k}$ fails to have locally bounded norm of the second fundamental form in $N\setminus\{q_1,q_2\}$, then
let $q_3\in N\setminus \{q_1,q_2\}$ be a point in $N$ that is furthest away from $\{q_1, q_2\}$
and such that, after passing to a subsequence,  there exists a sequence of
points $p_{3,k}\in M_{3,k}$ converging to $q_3$ with $\lim_{n\to\infty}A_{M_{k,3}}(p_{3,k})=\infty$.
Continuing inductively in this manner and using a diagonal-type argument, we obtain
after reindexing,
a new subsequence $M_k$ (denoted in the same way) and
a countable (possibly finite) non-empty set $\Delta':=\{q_1,q_2, q_3,\dots \}\subset N $
such that the following holds. For every $i\in \N$, there exists an integer $N(i)$ such that for all $k\geq N(i)$ there
exist points $p(k,q_i)\in M_k\cap B^N_{1/k}(q_i)$ where $A_{M_k}(p(k,q_i))>k$.
We let $\Delta$ denote the closure of $\Delta'$ in $N$. It follows
from the construction of $\Delta$ that
the sequence $M_n$ has locally bounded norm of the second fundamental form in $N\setminus \Delta$.

In light of the previous discussion, given a sequence $\{M_k\}_{k\in \N}$ of $H$-hypersurfaces in $N$, after possibly replacing it with a subsequence, we will consider $\Delta$ to be a well-defined singular set of convergence, as in Definition~\ref{def:singset}.

\begin{lemma}\label{curvest} 
Let $\{M_k\}_{k\in \N}$ be a sequence of $H$-hypersurfaces with $\sup_k \Ind_0(M_k) <\infty$ and assume that either $n=3$ or $n\le 7$ and for any open $B\In\In N$ there exists a constant $\mu_B$ such that $\sup_k\h^{n-1}(M_k\cap B)<\mu_B$.
Then, up to subsequence there exists a finite singular set of convergence $\Delta$ with $|\Dl|\leq \sup_k\Ind_0(M_k)+1$. Moreover, there exists a constant $C$ such that for any open $B\In \In N\setminus \Delta$
\[
\lim_{k\to \infty}\sup_{M_k\cap B}|A_{M_k}|\leq \frac{C}{\dist_N(B,\Delta)}.
\]
\end{lemma}
\begin{proof} The proof is similar to that of \cite[Claims 1 and 2]{Ben}. Let $I\in \N$ be such that $\Ind_0(M_k)+1\le I$ for all $k$ and assume that $\Delta $ has at least $I+1$ distinct points $\{q_1,\dots q_{I+1}\}$. Let
\[
\varepsilon<\frac12\min\{\min_{i\ne j} \dist_N(q_i, q_j), \sigma_N \}\,,
\]
where $\sigma_N$ is a lower bound for the injectivity radius of $N$. By Lemma \ref{sreg}, after passing to a subsequence, $\Ind(B_\ve^N(q_i)\cap M_k)>0$, for all $1\le i\le I+1$. Since $\{B_\ve^N(q_i)\}_{i=1}^{I+1}$ are pairwise disjoint we obtain  that $\Ind(M_k)\ge I+1$ and by Lemma \ref{lem:index} $\Ind_0(M_k)+1\ge I+1$,  which is a contradiction.

To prove the curvature estimate, it suffices to show that there exists $\eps_0>0$ and a subsequence (not re-labelled) so that for all $0<\ve\leq \eps_0$  
\begin{equation}\label{eq:stab}	
\lim_{k\to \infty}\Ind((B_\ve^N(q_i)\setminus B_{\ve/2}^N(q_i))\cap M_k)=0 \quad \text{for all $q_i\in \Delta$.}
\end{equation}
This is indeed sufficient, because $M_k$ has locally bounded norm of the second fundamental norm in $N\setminus \Delta$ and \eqref{eq:stab} combined with Lemma \ref{sreg} yields the required curvature estimate. 

To prove \eqref{eq:stab} we argue by contradiction: suppose there exists $q_i\in \Delta$ so that for all $\ve_0>0$, there exists $\eps_1\leq \eps_0$ with $\liminf \Ind((B_{\ve_1}^N(q_i)\setminus B_{\ve_1/2}^N(q_i))\cap M_k)\geq 1$. We can successively apply this statement (setting $\eps_0 = \eps_l/2$ for each later iteration) $I+1$ times to find a sequence $\eps_1, \eps_2, \eps_3, \dots \eps_{I+1}$  satisfying $\eps_{l+1}\leq \eps_l /2$ and $\liminf \Ind((B_{\ve_l}^N(q_i)\setminus B_{\ve_l/2}^N(q_i))\cap M_k)\geq 1$. Once again we have found $I+1$ disjoint sets for which each $M_k$ is unstable and shown $\Ind_0(M_k) \geq I+1$ for all large $k$, a contradiction.

\end{proof}

To study the limiting behaviour of CMC surfaces, we will need the following definition.

\begin{definition}\label{def:ee}
A connected subset $V\In N$ will be called an \emph{effectively embedded} $H$-hypersurface if $V$ is a finite union of smoothly immersed compact connected constant mean curvature hypersurfaces and at any point $p\in V$, there exists $\eps>0$ such that either
\begin{enumerate}
\item $B^N_\eps (p)\cap V$ is a smooth embedded disk, or 
\item $B^N_\eps (p)\cap V$ is the union of two embedded disks, meeting tangentially and whose mean curvature vectors point in opposite directions.
\end{enumerate}
\end{definition}

Let $V$ be an effectively embedded $H$-hypersurface as in Definition \ref{def:ee}.
We will refer to the set of points $p\in V$ satisfying $1.$ of Definition \ref{def:ee} as the regular part of $V$ and we will denote it by $e(V)$\footnote{$e(V)$ standing for the embedded part of $V$}. Note that $e(V)$ is relatively open and splits into a finite number of (mutually disjoint) connected components 
\[e(V)=\cup_{i=1}^L V^i,\] 
each of which is a smooth embedded CMC hypersurface having the same size mean curvature $H$. The set of points satisfying $2.$ of Definition \ref{def:ee} is the singular set of $V$, denoted by $t(V)$\footnote{$t(V)$ for touching set} which is relatively closed, 
and 
\[t(V):=\cup_{i=1}^L \ov{V}^i\setminus V^i.\] Notice that we cannot necessarily rule out $\ov{V}^i$ self-intersecting, however, with this notation we have that if $p\in t(V)$ then there exists $\eps >0$ so that $e(V)\cap B^N_\eps (p)$ splits into two disjoint components $C^i, C^j$ with $C^i\subset V^i$, $C^j\subset V^j$  and $\{\ov{C^i}\}_{i=1,2}$ are the two smooth embedded CMC disks touching tangentially at $p$ with opposite mean curvature vectors. It might happen that $i=j$ if one component $V^i$ self-intersects. 
It is not difficult to check that each $\ov{V}^i$ is individually an immersed, smooth, connected CMC hypersurface which is embedded unless it is self-intersecting.

Below is a definition of convergence that we will be using often in this paper and we will be refering to as \emph{H-convergence}.
\begin{definition}\label{def:conv_CMC}
A sequence $\{M_k\}_{k\in \N}$ of $H$-hypersurfaces \emph{$H$-converges}  to $V=\cup_{i=1}^L \ov{V}^i$, an effectively embedded $H$-hypersurface, with finite multiplicity $(m^1,\dots, m^L)\in \N^L$ if $d_\h (M_k,V)\to 0$ as $k\to\infty$ and if its singular set of convergence $\Delta\subset V$ is finite and whenever $p\in V\setminus \Delta$ the following holds.
\begin{itemize}
\item If $p \in V^i$, then there exists an $\eps >0$ so that $B^N_\eps (p)\cap M_k$ converges smoothly and graphically (normal grpahs) with multiplicity $m^i$, to $B_{\eps}^N(p)\cap V$.  	
\item If $p \in t(V)$, then there exists an $\eps >0$ so that $B^N_\eps (p)\cap M_k$ uniquely partitions into two parts. The first part converges smoothly and graphically, with multiplicity $m^i$, to $\ov C^i$, and the second converges smoothly and graphically, with multiplicity $m^j$, to $\ov C^j$, where $C^i$, $C^j$ are as discussed in the previous paragraph.
\end{itemize}
	
\end{definition}

\begin{remark}
If $\Delta=\emptyset$  then $V=\ov{V}^i$ for some fixed $i$ and the multiplicity of convergence is one, contrary to what happens if we allow the limit to be minimal\footnote{For instance in the standard $S^3 = \{x\in \R^4 : |x|=1\}$, if $S^2=\{x_4=0\}\cap S^3$ is a great sphere, the equidistant surfaces $M_k$ defined by $M_k=\{x_4:=1/k\}$ are CMC spheres converging smoothly to $S^2$. If we project this picture to $\R P^3$ then we have a sequence of CMC spheres converging smoothly (so $\Delta=\emptyset$) with multiplicity two to a great $\R P^2$.}. This follows from the fact that all $H$-hypersurfaces are two-sided. Thus over each $\ov{V}^i$ we can write the approaching $M_k$'s globally as graphs - if the multiplicity is larger than one, or there is more than one $\ov{V}^i$,  the $M_k$'s must have been disconnected.  

\end{remark}

Finally, in the next sections, we will also use the following notation. We let $S_0, I_0, V_0>0$ denote  a bound for the absolute sectional curvature, the injectivity radius and the volume of $N$. Given $H>0$, we fix $J_H\in(0,I_0)$ so that for any $\rho \leq J_H$, the geodesic balls $B^N_\rho(p)$ are $H$-convex, that is their boundaries are hypersurfaces whose mean curvature is bigger than or equal to $H$, independently of $p\in N$.

\section{Area estimate and compactness}\label{sec_area}
When $n=3$, we  use the results in Section~\ref{prelim} to prove the following area estimate for $H$-surfaces, $H>0$. 


\begin{thm}\label{thm_area}
Given $\mathcal I\in \mathbb N$ and $H>0$, there exists a constant $\mathcal{A}:=\mathcal{A}(\mathcal{\mathcal I}, N)$ such that if $M$ is an $H$-surface separating $N$ with $\Ind_0(M)\le \mathcal I$, we have that 
\[
\h^2(M)\leq \mathcal{A}.
\]
\end{thm}

\begin{proof}
We first prove a local area estimate when the norm of the second fundamental form is bounded.

\begin{claim}\label{area51}
Given $\alpha>0$ there exists $\omega:=\omega(\alpha,N)$
such that the following holds. Given $p\in M$ and $\rho<J_H$, if $\sup_{B^N_\rho(p)}|A|<\alpha$ then 
\[
\h^2(M\cap B^N_{\rho\slash 2}(p))< \omega \h^3(N).
\]
\end{claim}
\begin{proof}[Proof of Claim \ref{area51}] Given $\rho<J_H$, the techniques used to prove Lemma 3.1 in~\cite{mt3} give that there exists $\beta:=\beta(\alpha,J_H, S_0)>0$ such that if $M\cap B^N_\rho(p)$ bounds an $H$-convex domain, then $M\cap B^N_{\rho\slash 2}(p)$ has a one-sided regular neighbourhood of fixed size $\beta$.  This means that the collection of geodesics of length $\beta$ starting at $x\in M\cap B^N_{\rho\slash 2}(p)$ and with initial velocity given by $H(x)/|H(x)|$ are pairwise-disjoint, only intersect $M$ at $x$ and therefore foliate a one-sided neighbourhood of $M$. The result is mainly a consequence of the observation that two $H$-surfaces with bounded norm of the second fundamental form which are oppositely oriented and such that one lies on the mean convex side of the other, cannot be too close away from their boundary and this is essentially a consequence of the maximum principle for quasi-linear uniformly elliptic PDEs. Note that this is not true when $H=0$.

Since $M$ is separating in $N$ we do have that $M\cap B^N_\rho(p)$ bounds an $H$-convex domain. 
Let $\mathcal U_\beta$ denote the 1-sided regular neighbourhood of $M\cap B^N_{\rho\slash 2}(p)$ as above. Then,  since the norm of second fundamental form of $M$ is uniformly bounded we can directly relate the area of $M\cap B^N_{\rho\slash 2}(p)$ with the volume of $\mathcal U_\beta$: there exists a constant $\omega:=\omega(\beta)>0$ such that 
\[
\frac{1}\omega \h^2(M\cap B^N_{\rho\slash 2}(p))\leq \h^3(\mathcal U_\beta)\leq \h^3(N).
\]
This finishes the proof of the claim.  
\end{proof}

We now begin the proof of the area estimate. Arguing by contradiction, assume that there exist $\mathcal I\in \mathbb N$, $H>0$, and a sequence of $H$-surfaces $\{M_k\}_{k\in \N}$ such that for all $k\in \N$,  the $H$-surface $M_k$ separates $N$, $\Ind_0(M_k)\le \mathcal I$ and 
\[
\h^2(M_k)>k.
\]

By Lemma \ref {curvest}, after passing to a subsequence, there exists a finite set of points $\Delta:=\{p_1, \dots, p_l\}$, $l\le  \mathcal I+1$, such that the sequence $\{M_k\}_{k\in \N}$  has
 locally bounded norm of the second fundamental form in $N\setminus \Delta$. Since $N$ is compact, applying Claim~\ref{area51} and a covering argument gives that for any $\ve>0$, there exists a constant $V(\ve)$ such that 
\[
\h^2(M_k\cap [N\setminus \bigcup_{i=1}^lB^N_{\ve}(p_i)])<V(\ve).
\]
In order to obtain a contradiction, it remains to show that the area of $M_k\cap B^N_{\ve}(p_i)$, $i=1,\dots, l$ is also bounded, uniformly in $k$. To that end, we will use the monotonicity formula for the area. After isometrically embedding the ambient space $N$ in an Euclidean space $\R^m$, the submanifolds $M_k\subset N\subset \R^m$ have  mean curvature vector fields $H_k=H^N_k+H^{N^\perp}_k$, where $H_k^N$ and $H^{N^\perp}_k$ are the projections of $H_k$ (the mean curvature vector of $M_k\subset \R^m$) onto the tangent and the normal space of $N$ respectively. Note that $|H^N_k|=H$ and $H^{N^\perp}_k$ depends only on the embedding of $N$ and thus its norm is uniformly, in $k$, bounded. We thus have a sequence of submanifolds with uniformly bounded mean curvature, $|H_k|\le c$. Therefore, the area  monotonicity, see for example \cite[17.6]{si1}, yields, for any $p\in \R^m$ and $0<\s<\r$,
\[
e^{c\s}\s^{-2}\h^2(M_k\cap\{x:|x-p|<\sigma\})\le e^{c\r}\r^{-2}\h^2(M_k\cap\{x:|x-p|<\r\})\,.
\]
Since $M_k\subset N$ and the embedding is isometric we obtain
\[
e^{c\s}\s^{-2}\h^2(M_k\cap B^N_\s(p))\le e^{c\r}\r^{-2}\h^2(M_k\cap B^N_\r(p)).
\]
Take now $p$ to be a point in the singular set. Then for small $\ve$ we have
\[
\ve^{-2}\h^2(M_k\cap B^N_\ve(p))\le e^{c\ve}(2\ve)^{-2}\h^2(M_k\cap B^N_{2\ve}(p))\le\frac12\ve^{-2}\h^2(M_k\cap B^N_{2\ve}(p))\,,
\]
which yields
\[
\h^2(M_k\cap B^N_\ve(p))\le \h^2(M_k\cap (B^N_{2\ve}(p)\setminus B^N_\ve(p))).
\]
But now, choosing $\ve$ small enough so that $B^N_{2\ve}(p)\setminus B^N_\ve(p)$ is away from $\Delta$, the right hand side is uniformly bounded by $V(\ve)$ and thus $\h^2(M_k)< (l+1)V(\ve)$.  This contradicts the assumption that $\h^2(M_k)>k$ and finishes the proof of the area estimate.
\end{proof}

\begin{remark}\label{example}
In~\cite{tra5}, Traizet proved for any positive integer $g$, $g\neq 2$, every flat
3-torus  admits connected closed embedded and separating 
minimal surfaces of genus $g$  with arbitrarily large area. Fix $g\neq 2$ and let $M_k$ be a sequence of such minimal surfaces whose area is becoming arbitrarily large. Since the genus is fixed, by the Gauss-Bonnet theorem, the total curvature of $M_k$ is uniformly bounded in $k$. And this gives that the index of $M_k$ is also uniformly bounded in $k$~\cite{ty}. Thus, these examples show that the area estimates do not hold when $H=0$.
\end{remark}

As a corollary of the proof above, if the ambient manifold $N$ has finite fundamental group (e.g. if it has positive Ricci curvature), then the area bound is true without assuming that the $H$-surface $M$ is separating.

\begin{corollary}\label{cor_area}
Given $\mathcal I\in \mathbb N$ and $H>0$, there exists a constant $\mathcal{A}:=\mathcal{A}(\mathcal{\mathcal I}, N)$ such that if $M$ is an $H$-surface in $N$ with $\Ind_0(M)\le \mathcal I$ and $N$ has finite fundamental group, we have that 
\[
\h^2(M)\leq \mathcal{A}.
\]
\end{corollary}
\begin{proof}
Since $N$ has finite fundamental group its universal cover $\Pi:\widetilde N\to N$ is a finite covering. $\Pi^{-1}(M)$ is a disjoint collection of $H$-hypersurfaces in $\widetilde{N}$ and we denote by $\widetilde M$ a connected component of $\Pi^{-1}(M)$. Then $\widetilde M$ is an $H$-surface separating $\widetilde N$, because  $\widetilde N$ is simply-connected. We may now reduce to the setting of Theorem~\ref{thm_area}: let $\{M_k\}\In N$ be a sequence of $H$-hypersurfaces with index uniformly bounded by $\mathcal{I}$. By Lemma \ref {curvest}, after passing to a subsequence, there exists a finite set of points $\Delta:=\{p_1, \dots, p_l\}$, $l\le  \mathcal I+1$, such that the sequence $\{M_k\}_{k\in \N}$  has
 locally bounded norm of the second fundamental form in $N\setminus \Delta$. Thus picking connected lifts $\widetilde{M}_k \In \widetilde{N}$ we have that $\widetilde{M}_k$ are separating and there exists a finite set of points $\widetilde\Delta:=\{\widetilde{p}_1, \dots, \widetilde{p}_L\}$, $L\le  |\pi_1(N)|(\mathcal I+1)$, such that the sequence $\{\widetilde{M}_k\}_{k\in \N}$  has
 locally bounded norm of the second fundamental form in $\widetilde{N}\setminus \widetilde\Delta$. We can now apply Claim \ref{area51} to $\widetilde{M}_k\In \widetilde{N}$ and follow the remaining parts of the proof of Theorem \ref{thm_area} to conclude the proof of the corollary.  \end{proof}

Thanks to the area estimate, an elegant compactness result for $H$-surfaces separating $N$ now follows. 

 \begin{theorem}\label{thm:conv} Given $H>0$, let $\{M_k\}_{k\in \N}$ be a sequence of $H$-surfaces such that, for all $k\in \N$, $M_k$ separates $N$ (or not necessarily separating if $|\pi_1(N)|<\infty$) and $\sup_k \Ind_0(M_k)<\infty$. 
 Then, there exists an effectively embedded $H$-surface $M_\infty$  such that, after passing to a subsequence, $\{M_k\}_{k\in \N}$ H-converges with multiplicity one to $M_\infty$, where the convergence is as in Definition \ref{def:conv_CMC}.
 

 \end{theorem}
\begin{proof}
Using the curvature estimate of Lemma \ref{curvest} and the area estimate of Theorem~\ref{main} (or Corollary~\ref{cor_area} if $N$ has finite fundamental group),
  a standard argument yields that away from a finite set of points  $\Delta\subset N$, that is the singular set of convergence (see Definition~\ref{def:singset}), a subsequence $H$-converges with finite multiplicity to a surface $M_\infty$ effectively embedded in $N\setminus\Delta$ with constant mean curvature $H$. 

We next show that $M_\infty\cup \Delta$ is in fact effectively embedded in $N$, which will imply that $\{M_k\}_{k\in \N}$ $H$-converges with finite multiplicity to $M_\infty\cup\Delta$ with $\Delta$ being the singular set of convergence. For this we will need the following claim. We let $\Delta=\{q_1,\dots q_l\}$ and $\ve:=\frac 12 \inf_{i,j=1,\dots, l; i\neq j} \dist_N(q_i,q_j)$.

\begin{claim}\label{dcurvest}
Given $\delta>0$, there exists $0<\rho\leq \ve$ such that for any $q_i\in \Delta$ and $p\in M_\infty\cap B_\rho^N(q_i)$ 
\[
|A_{M_\infty}|(p)\leq \frac{\delta}{\dist_N(p,q_i)}\,.
\]
\end{claim}
\begin{proof}[Proof of Claim \ref{dcurvest}]
Note first that, by the nature of the convergence and Lemma~\ref{curvest},  for any $q_i\in \Delta$ and $p\in M_\infty\cap B^N_\ve(q_i)$ we have
\begin{equation}\label{eq:CEst}
|A_{M_\infty}|(p)\leq \frac{C}{\dist_N(p,q_i)}.
\end{equation}
Moreover, arguing as in \cite[Claim 2]{Ben} taking $\ve$ even smaller if necessary we have that 
each connected component of $M_\infty\cap (B^N_\ve(q_i)\setminus\{q_i\})$ for all $q_i\in \Delta$ is strongly  stable.

To prove the claim we argue by contradiction and suppose that for some $\de>0$ there exist $q\in \Delta$ and a sequence of points $p_k\in M_\infty$ such that $\lim_{k\to\infty}p_k=q$ and 
\[
|A_{M_\infty}|(p_k)> \frac{\de}{\dist_N(p_k,q)}\,.
\]
Consider now scaling $M_\infty$ by $\frac{1}{\dist_N(p_k,q)}$, with the scaling performed in geodesic coordinates and with origin at $q$. Letting $k\to \infty$, and since $\dist_N(p_k,q)\to 0$,  after passing to a subsequence, the scaled surfaces converge to a tangent cone of $M_\infty\cup \{q\}$ at $q$. The convergence is in general weak convergence, however, by the curvature estimate \eqref{eq:CEst} and the comments following it, it is in fact smooth away from the origin and the limit is strongly stable away from the origin. Since the limit is also a stationary cone it must be a plane. This contradicts the fact that there exists a point at distance 1 from the origin with $|A|\ge \de>0$.
\end{proof}

We can now show that $M_\infty\cup \Delta$ is effectively embedded following the ideas of \cite{wh5} (see also \cite[Theorem 4.3]{Ao}).
Let $p\in \Delta$ and $r>0$ be such that $B^N_{2r}(p)\cap \Delta=\{p\}$. Consider a sequence $r_i\to 0$ and denote by $\widetilde M_i$ the scaling of $M_\infty \cap B^N_r(p)$ by $1/ r_i$. Then, the curvature estimates of Claim \ref{dcurvest} yield that,  after passing to a subsequence $\widetilde M_i$ converge to a union of planes. This in turn implies that $M_\infty$ is a union of disks and punctured disks. We can thus argue exactly as in \cite[Theorem 4.3]{Ao} to show that  $M_\infty\cup\Delta$ is indeed effectively embedded.

That the multiplicity of convergence is 1 will be a consequence of the results in Section \ref{section:multiplicity}.
\end{proof}

 The curvature estimates discussed in Section~\ref{prelim} and that were used to prove Theorem~\ref{main} and Theorem~\ref{thm:conv}, crucially depend on a bound for the volume of the $H$-hypersurface when $3<n\leq 7$,. However, if one assumes an a priori volume bound, then the  proof of Theorem~\ref{thm:conv}  can be modified to prove a compactness result in higher dimensions, that is Theorem~\ref{thm:conv2} below.  As in Theorem~\ref{thm:conv}, multiplicity 1 will be a consequence of the results in Section \ref{section:multiplicity}. 
 
 \begin{thmn}[\ref{thm:conv2}] Given $H>0$, let $\{M_k\}_{k\in \N}$ be a sequence of $H$-hypersurfaces in $N$ satisfying
 \[
 \sup_k \h^{n-1}(M_k)<\infty \quad \text{and}\quad \sup_k \Ind_0(M_k)<\infty.
 \]
 Then, there exists a hypersurface $M_\infty$ effectively embedded in $N$ with constant mean curvature $H$, such that, after passing to a subsequence, $\{M_k\}_{k\in \N}$ H-converges with multiplicity one to $M_\infty$, where the convergence is as in Definition \ref{def:conv_CMC}.  

 \end{thmn}

\section{Multiplicity analysis}\label{section:multiplicity}

The main goal of this section, is to show that under certain hypotheses, a sequence of $H$-hypersurfaces that converges to an effectively embedded surface, will in fact  converge with multiplicity one to its limit. This result will complete the proofs of Theorems \ref{thm:conv} and \ref{thm:conv2}.

We first recall that $I_0>0$ denotes  a bound for the injectivity radius of $N$. And that given $H>0$, we have fixed $J_H\in(0,I_0)$ so that for any $\rho \leq J_H$, the ambient geodesic balls $B^N_\rho(p)$ are $H$-convex, independently of $p\in N$. Throughout this section, we will always assume that the radius of an ambient geodesic ball is less than $J_H$.

We will show that even if $\Delta\neq \emptyset$ we must always have multiplicity one convergence: 

\begin{theorem}\label{thm:mult}
Let  $V=\cup_{\ell =1}^L \ov{V}^\ell$ be a hypersurface effectively embedded in $N$ with constant mean curvature $H>0$ and  let $\{M_k\}_{k\in \N}$ be a sequence of $H$-hypersurfaces that H-converges to $V$ with multiplicity $(m^1,\dots, m^L)\in \N^L$. Then the singular set of convergence $\Delta$ lies inside $t(V)$ and $m^\ell = 1$ for all $\ell=1, \dots, L$. \end{theorem}

\begin{proof}
Since $M_k$ is embedded with uniformly bounded volume and the number of points in $\Delta$ is finite, there exist $0<2\eps<\dl<J_H$ such that for $k$ sufficiently large and $y\in\Delta$, $B_\dl^N(y)\setminus B_\eps^N(y)\cap M_k$ is a collection of $m(y)\geq 1$ graphs  of functions $u_i^y$, $i:=1,\dots, m(y)$, over $V$ which converge smoothly to zero in $k$ (where for simplicity we have omitted the index $k$). If $y\notin t(V)$, let $n_y=H/|H|$ be the unit normal to $V$ at $y$, otherwise let $n_y$ be a choice of unit normal. The graphs of $u_i^y$, $i:=1,\dots ,m(y)$, converge smoothly to $B_\dl^N(y)\setminus B_\eps^N(y)\cap V$ as $k\to \infty$ and can be  ordered by height, say with respect   to $n_y$, so that $u_i^y$ is \emph{above} $u_{i+1}^y$ for $i:=1,\dots ,m(y)-1$.  Let $Q_i^y$ be the connected component of $B_\dl^N(y)\cap M_k$ that contains $\graph u_i^y$.

\begin{claim}\label{Dtv}
$\Delta\subset t(V)$.
\end{claim}
\begin{proof}[Proof of Claim \ref{Dtv}]
 Arguing by contradiction, suppose that $y\in \Delta\cap e(V)$ - so that $y$ lies on an embedded part of the limit. Then, by definition, $V\cap B_\dl^N(y)\subset V^\ell$ is an embedded CMC disc and the collection of $\graph u_i^y$, $i:=1,\dots ,m(y)$, converges to $V\cap [B_\dl^N(y)\setminus B_\eps^N(y)]$.  
 
 If for all $i:=1,\dots, m(y)$, $Q_i^y\cap [B_\dl^N(y)\setminus B_\eps^N(y)]=\graph u_i^y$ (i.e. $Q_i^y\cap [B_\dl^N(y)\setminus B_\eps^N(y)]$ is connected)  then since $Q_i^y$ converges to the disc $V\cap B_\dl^N(y)$  as Radon measures with multiplicity one, by Allard's regularity theorem \cite{al1} the convergence is smooth and $y\not \in \Delta$. Therefore, there exists $i\in\{1,\dots m(y)\}$, such $Q_i^y\cap [B_\dl^N(y)\setminus B_\eps^N(y)]$ consists of more than one connected components. However, note that because $Q_i^y$ separates $B_\dl^N(y)$, the sign of the inner product between the unit normal to $Q_i^y$ and $n_y$ must change as we alternate components of $Q_i^y\cap [B_\dl^N(y)\setminus B_\eps^N(y)]$. This contradicts the fact that such components must converge to a single CMC disc $V\cap B_\dl^N(y)$. This contradiction proves that $\Delta\subset t(V)$.
\end{proof}

It remains to prove that the convergence to $V$ is with multiplicity one. Let $y\in \Delta\subset t(V)$, then $B^N_\dl (y)\cap V$ is the union of two embedded discs, $C^\pm$  meeting tangentially and whose mean curvature vectors point in opposite directions. Without loss of generality, we pick $n_y=H^+/|H^+|$ where $H^+$ is the mean curvature of $C^+$ and thus so that $C^+$ lies above $C^-$, in the sense discussed in the first paragraph of the proof. The collection $\graph u_i^y$, $i:=1,\dots ,m(y)$, converging smoothly to $B_\dl^N(y)\setminus B_\eps^N(y)\cap V$ as $k\to \infty$ can be divided into two distinct finite collections of graphs $\Delta_+$ and $\Delta_-$ that satisfy the following properties:
\begin{itemize}
\item the graphs in $\Delta_+$ are above the graphs in $\Delta_-$; 
\item the collection $\Delta_+:=\{\graph u_{i,+}^y\,,\,\, i:=1,\dots, m_+(y)\}$, converges smoothly to $C^+\cap [B_\dl^N(y)\setminus B_\eps^N(y)]$ as $k\to \infty$;
\item the collection $\Delta_-:=\{\graph u_{i,-}^y\,,\,\ i:=m_+(y)+1,\dots m_-(y)\}$, converges smoothly to $C^-\cap [B_\dl^N(y)\setminus B_\eps^N(y)]$ as $k\to \infty$.
\end{itemize}

Recall that $Q_i^y$ is the connected component of $B_\dl^N(y)\cap M_k$ that contains $\graph u_i^y$. Just like we observed before, if $Q_i^y\cap [B_\dl^N(y)\setminus B_\eps^N(y)]$
consists of more than one connected component, since $Q_i^y$ separates $B_\dl^N(y)$, then the sign of the inner product between the unit normal to $Q_i^y$ and $n_y$ must change as we alternate component of $Q_i^y\cap [B_\dl^N(y)\setminus B_\eps^N(y)]$. This implies that alternating components must alternating convergence to $u^y_+$ and $u^y_-$. This gives that if $Q_i^y\cap [B_\dl^N(y)\setminus B_\eps^N(y)]$
consists of more than one connected component, then it consists of exactly two components, one in $\Delta_+$ and the other in $\Delta_-$. And $Q_i^y$ converges to $B^N_\dl (y)\cap  V$ on compact subsets of $B^N_\dl (y)\setminus \{y\}$ with multiplicity 1.

\begin{claim}\label{disconn}
There is only one $Q_i^y$ such that $Q_i^y\cap [B_\dl^N(y)\setminus B_\eps^N(y)]$ is disconnected.
\end{claim}
\begin{proof}[Proof of Claim \ref{disconn}]
 Arguing by contradiction, assume that $Q_j^y$, $i\neq j$ also has the property that $Q_j^y\cap [B_\dl^N(y)\setminus B_\eps^N(y)]$ consists of exactly two components. Let $Q_i^y\cap [B_\dl^N(y)\setminus B_\eps^N(y)]= \graph u_{i,+}^y\cup \graph u_{l_i,-}^y$ and let $Q_j^y\cap [B_\dl^N(y)\setminus B_\eps^N(y)]= \graph u_{j,+}^y\cup \graph u_{l_j,-}^y$. Then, because of the convergence and separation properties, we can assume that $j<i<l_i<l_j$. 

Let $W$ be the connected component of $B_\dl^N(y)\setminus Q_i^y\cup Q_j^y$ such that $Q_i^y\cup Q_j^y\subset \partial W$. The convergence and elementary separation properties yield that the mean curvature vector of $M_k$ is pointing outside $W$ on $Q_j^y$ and inside $W$ on $Q_i^y$.  Moreover, as $k\to \infty$, we have that $\overline W\to C^+\cup  C^-$ in Hausdorff distance. The argument described in~\cite{alr1} can be modified to prove the following claim. 

\begin{claim}\label{stableinw}
If $Q_i^y$ is not strongly stable, then there exists a compact, oriented, stable hypersurface $\Gamma$ embedded in $W$ with constant mean curvature $H$ and such that
$\partial \Gamma = \partial Q_i^y$ and
$\Gamma $ is homologous to $Q_i^y$ in $W$.
\end{claim}

\begin{proof} [Proof of Claim \ref{stableinw}]
Let $\mathcal F$ be the family of subsets $Q\subset W$ of finite perimeter whose boundary $\partial Q$  is a rectifiable integer multiplicity current such that $Q_i^y\subset \partial Q$ and let $\Sigma=\partial Q\setminus Q^y_i$, so that  $\partial \Sigma=\partial Q_i^y$. Given $\mu>0$, 
let $F_\mu\colon \mathcal F\to \mathbb R$ be the functional
\[
F_\mu(Q)=\h^{n-1}(\Sigma)+(H+\mu)\h^n(Q)\,.
\]

Let $W_1$ be the mean convex component of $B_\dl^N(y)\setminus Q_i^y$, let $S_{min}\subset W_1$ be a volume minimizing hypersurface  with $\partial 
S_{min}=\partial Q_i^y$ and homologous to $Q_i^y$, and let $Q_{min}$ denote the region in $W_1$ enclosed by $Q_i^y\cup S_{min}$~\cite{DG61, FF72, Giusti}. Recall that since $n\leq 7$, no singularities occur. 

Let   $Q_\rho:=\{x\in W : \dist_{N}(x, Q_j^y)\leq \rho\}$ and note that if $\rho$ is chosen sufficiently small, then the sets 
\[
S_t:=\{x\in Q_\rho \text{ such that } \dist_{N}(x, Q_j^y)=t \}, \; 0\leq t \leq \rho
\]
are smooth hypersurfaces parallel to $Q_j^y$ and foliating $Q_\rho$. Let $Y$ be the the unit vector field normal to the foliation and pointing toward $Q_j^y$. Let $H_t$ denote the mean curvature of $S_t$ as it is oriented by $Y$. Then 

\[
\frac{d}{dt}H_t\mid_{t=0}=|A|^2 + \Ric_N(n_j,n_j)
\]
where $n_j$ is the unit normal vector field to $Q_j^y$. Thus, for any $\mu>0$ there exists $\rho_\mu>0$, depending on $\Ric_N(n_j,n_j)$, such for $t\in [0,{\rho_\mu}]$  we have that $H_t< H+\mu$ and at a point $p\in S_t$
\[
\dvg_N Y=\dvg_{S_t} Y= -H_t \quad \implies \quad -H-\mu< \dvg_{N} Y.
\]
Let $Q_{par}:=Q_{\rho_\mu}$ and $S_{par}= S_{\rho_\mu}$.

Next we are going to work on $Q_i^y$. Let $\phi$ be the first eigenfunction of the stability operator of  $Q_i^y$. The eigenfunction $\phi$ is positive in the interior of $Q_i^y$ and since $Q_i^y$ is not stable, then 
\[
\Delta \phi +|A|^2 \phi +\Ric_N(n_i,n_i)\phi+\lambda_1\phi = 0, 
\]
where $\lambda_1$ is a negative constant and $n_i$ is the unit normal vector field to $Q_i^y$. And possibly after a small perturbation of $\dl$, we can assume that $0$ is not an eigenvalue of $\Delta  +|A|^2  +\Ric_N(n_i,n_i)$. Thus there is a smooth function $v$ vanishing on $\partial Q_i^y$, such that $\Delta v  +|A|^2 v +\Ric_N(n_i,n_i) v=1$ in $Q_i^y$. By Hopf's maximum principle the derivative of $\phi$ with respect to the outer pointing normal vector to $\partial Q_i^y$ is strictly negative. Therefore, there exists $a>0$ small, such that $u=\phi+av$ is positive in the interior of $ Q_i^y$.

Let 
\[
\widetilde S_t:=\{x\in W \text{ such that } \dist_{N}(x, Q_i^y)=t u \}, \; 0\leq t \leq \widetilde \rho.
\]
If $\widetilde \rho$ is sufficiently small, the sets $\widetilde S_t$ are smooth hypersurfaces foliating a closed neighbourhood $\widetilde Q_{\widetilde \rho} $ of $Q_i^y$ in $W$. 

Let $X$ be the the unit vector field normal to the foliation and pointing away from $Q_i^y$. Let $H_t$ denote the mean curvature of $\widetilde S_t$ as it is oriented by $X$. Then 
\[
\frac{d}{dt}H_t\mid_{t=0}=\Delta u +|A|^2u + \Ric_N(n_i,n_i)u=-\lambda_1\phi +a>0,
\]
where $n_i$ is the unit normal vector field to $Q_i^y$. Therefore, if $\widetilde \rho$ is taken sufficiently small, for $t\in (0,\widetilde \rho]$ we have that $H_t> H$ and at a point $p\in \widetilde S_t$ we have
\[
\dvg_{N} X<-H.
\]
Let $Q_{uns}:=\widetilde Q_{\widetilde \rho}$ and $S_{uns}:=\widetilde S_{\widetilde \rho}$.

\begin{claim}\label{subset}
Let $Q\in\mathcal F$ with $\Sigma$ smooth and transverse to $S_{min}, S_{par}$, and $S_{uns}$. The following statements hold.
\begin{enumerate}
\item If $Q\not\subset Q_{min}$ then  $F_\mu (Q\cap Q_{min})\leq F_\mu (Q)$;  
\item  If $Q\cap Q_{par}\neq \emptyset$ then $F_\mu (Q\setminus Q_{par})\leq F_\mu (Q)$;
\item If $Q_{uns}\not \subset Q$ then $F_\mu (Q\cup Q_{uns})\leq F_\mu (Q)$. 
\end{enumerate}
\end{claim}

\begin{proof}[Proof of Claim \ref{subset}]
We first prove that  if $Q\not\subset Q_{min}$ then  $F_\mu (Q\cap Q_{min})\leq F_\mu (Q)$. Since $Q\cap Q_{min}\subset Q$, we have that $\h^n(Q\cap Q_{min})\leq \h^n(Q)$ and, by construction, $\h^{n-1}(\Sigma')\leq \h^{n-1}(\Sigma)$ where $\Sigma':= \partial (Q\cap Q_{min})\setminus Q_i^y$.

We now prove that  if $Q\cap Q_{par}\neq \emptyset$ then $F_\mu (Q\setminus Q_{par})\leq F_\mu (Q)$. Recall that in $Q_{par}$, $-H-\mu< \dvg_N Y$, therefore 
\[
(-H-\mu)\h^n(Q\cap Q_{par})<\int_{Q\cap Q_{par}} \dvg_{N} Y=\int_{\partial (Q\cap Q_{par})} Y\cdot \nu
\]
where $\nu$ is the outer pointing unit normal to $\partial (Q\cap Q_{par})$ and
\[
\int_{\partial (Q\cap Q_{par})} Y\cdot \nu = \int_{Q\cap S_{par}} Y\cdot \nu+\int_{\Sigma \cap Q_{par}} Y\cdot \nu.
\]
Since, by construction, $Y\cdot \nu=-1$ on $S_{par}$ and $Y\cdot \nu\leq 1$ on $\Sigma \cap Q_{par}$, we have that
\[
(-H-\mu)\h^n(Q\cap Q_{par})< -\h^{n-1}(Q\cap S_{par})+\h^{n-1}(\Sigma \cap Q_{par})
\]
and
\[
F_\mu (Q\setminus Q_{par})= (H+\mu)(\h^n(Q)-\h^n(Q\cap Q_{par}))+\h^{n-1}(\Sigma\setminus Q_{par}) +\h^{n-1}(Q\cap S_{par})
\]
\[
< (H+\mu)\h^n(Q)+\h^{n-1}(\Sigma \cap Q_{par})+\h^{n-1}(\Sigma\setminus Q_{par})=F_\mu (Q)
\]

We finally prove that if $Q_{uns}\not \subset Q$ then $F_\mu (Q\cup Q_{uns})\leq F_\mu (Q)$. We argue similarly to the previous claim. Recall that in $Q_{uns}$, $\dvg_{N} X<-H$. Therefore 
\[
-H\h^n(Q_{uns}\setminus Q)>\int_{Q_{uns}\setminus Q} \dvg_{N} X=\int_{\partial (Q_{uns}\setminus Q)} X\cdot \nu
\]
where $\nu$ is the outer pointing unit normal to $\partial (Q_{uns}\setminus Q)$ and
\[
\int_{\partial (Q_{uns}\setminus Q)} X\cdot \nu = \int_{S_{uns}\setminus Q} X\cdot \nu+\int_{\Sigma \cap Q_{uns}} X\cdot \nu.
\]
Since, by construction, $X\cdot \nu=1$ on $S_{uns}$ and $X\cdot \nu\geq -1$ on $\Sigma \cap Q_{uns}$, we have that
\[
-H\h^n(Q_{uns}\setminus Q)> \h^{n-1}(S_{uns}\setminus Q)-\h^{n-1}(\Sigma \cap Q_{uns})
\]
and
\[
F_\mu (Q\cup Q_{uns})= (H+\mu)(\h^n(Q)+\h^n(Q_{uns}\setminus Q))+\h^{n-1}(\Sigma\setminus Q_{uns}) +\h^{n-1}(S_{uns}\setminus Q)
\]
\[
< (H+\mu)\h^n(Q)+\mu \h^n(Q_{uns}\setminus Q)+\h^{n-1}(\Sigma \cap Q_{uns})+\h^{n-1}(\Sigma \setminus Q_{uns})<F_\mu (Q).
\]

This finishes the proof of Claim~\ref{subset}.
\end{proof}

In order to find a minimizer for the functional $F_\mu$ we consider a minimizing sequence $Q_m$ and, since they have uniformly bounded areas, we can apply the compactness results of \cite{FF72} to extract a converging subsequence.  Note that by Claim~\ref{subset}, we can assume that $Q_m\subset Q_{min}$, $Q_m\cap Q_{par}= \emptyset$, and $Q_{uns}\subset Q_m$. It is known that a minimizer of $F_\mu$ is smooth \cite{Alm68, Bomb78, SS82NP} and thus we obtain a compact, embedded, oriented minimizer $\Gamma_\mu \subset W$ of the functional $F_\mu$ such that
$\partial \Gamma_\mu = \partial Q_i^y$ and
$\Gamma_\mu $ is homologous to $Q_i^y$ in $W$. In particular, $\Gamma_\mu$ has constant mean curvature equal to $H+\mu$. 

We can also assume that $H+\mu<2H$ and 
\[
\h^{n-1}(\Gamma_\mu)\leq \h^{n-1}(Q_i^y)\leq 2\h^{n-1} (C^{+}\cup C^{-}).
\]
The first inequality above follows because $\h^{n-1}(\Gamma_\mu)\leq F_\mu(\Gamma_\mu)\leq F_\mu (Q_i^y)= \h^{n-1}(Q_i^y)$. The second inequality holds because away from the singular point of convergence $y$, the volume can be bounded by the volume of the limit, and nearby $y$ it can be bounded by using the monotonicity formula for the volume, exactly like we have done to finish the proof of Theorem~\ref{main}. 
Then the results in \cite{bcw18} (see Lemma~\ref{sreg} and Remark~\ref{sregrem}) give that $\Gamma_\mu$ has norm of the second fundamental form uniformly bounded on compact sets of $B_\dl^N(y)$. And taking the limit of $\Gamma_\mu$ as $\mu$ goes to zero, we obtain in the limit the desired $\Gamma$ and finish the proof of Claim~\ref{stableinw}.
\end{proof}

We can now finish the proof of Claim \ref{disconn}. 
Since $y$ is a singular point of convergence, $Q_i^y$ cannot be strongly stable and thus cannot have norm of the second fundamental form bounded nearby $y$. Therefore Claim~\ref{stableinw} gives a compact, oriented, stable hypersurface $\Gamma$ embedded in $W$ with constant mean curvature $H$ and such that
$\partial \Gamma = \partial Q_i^y$ and
$\Gamma $ is homologous to $Q_i^y$ in $W$.

We now recall that while we have omitted the index $k$, we have in fact a sequence of domains $W(k)$ and stable hypersurfaces $\Gamma(k)\subset W(k)$.   By the previous discussion, $\Gamma(k)$ has norm of the second fundamental form uniformly bounded on compact sets of $B_\dl^N(y)$, uniform in $k$. And by construction, since $\Gamma(k)$ is homologous to $Q_i^y$ in $W$, for any $\rho>0$ there exists $k>0$ such that $\Gamma (k) \cap B_\rho^N(y)\neq \emptyset$. Using the uniform bound on the norm of the second fundamental form gives that $\Gamma(k)$ must converge smoothly to $C^+$ or $C^-$  or both. Elementary separation properties give that $Q_j^y$ cannot converge smoothly to $[C^+\cup C^- ]\setminus \{y\}$. This contradiction proves that there is only one $Q_i^y$ such that $[Q_i^y\cap B_\dl^N(y)]\setminus B_\eps^N(y)$ is disconnected.
\end{proof}

We now prove that the convergence to $V$ is with multiplicity one and finish the proof of the theorem. Arguing by contradiction, assume that the multiplicity of convergence along some $V^\ell$ is $m^\ell\geq 2$.  Recall that the convergence is smooth on compact subsets $K\subset \subset \ov{V}^\ell\sm \Delta$.  Observe that we must have $\Delta\cap \ov{V}^\ell\neq \emptyset$: if not, since $\ov{V}^\ell$ is connected, we can write the approaching $M_k$'s globally as graphs over $\ov{V}^\ell$ (since CMC hypersurfaces are always two-sided). And if there were more than one graph, then the $M_k$'s are disconnected.

Let $\Delta\cap \ov{V}^\ell=\{y_1, \dots, y_{g(\ell)}\}$. Since the convergence is smooth on $\ov{V}^\ell \setminus \bigcup_{j=1}^{g(\ell)} B^N_{\eps}(y_j)$ and with finite multiplicity, we can write the approaching surfaces $M_k \setminus \bigcup_{j=1}^{g(\ell)} B^N_{\eps}(y_j)$  globally as graphs over $\ov{V}^\ell \setminus \bigcup_{j=1}^{g(\ell)} B^N_{\eps}(y_j)$ and order such graphs by height with respect to the mean curvature vector $\vec{H}^\ell$ of $\ov{V}^\ell$. This gives ordered sheets $S_k^1,\dots ,S_k^{m^\ell}$ each converging smoothly to $\ov{V}^\ell \setminus \bigcup_{j=1}^{g(\ell)} B^N_{\eps}(y_j)$. Note that this ordering is different from the previous local ordering established nearby a singular point.   Let $y_j\in \Delta\cap \overline{V}^\ell\subset t(V)$ and recall that $V\cap [B^N_{\dl}(y_j)\sm B^N_{\eps}(y_j)]$ consists of two oppositely oriented components  which we denote by $\Gamma_j^+$ and $\Gamma_j^-$. Assume that $\Gamma_j^+\subset V^\ell$ and let $Q^1_j$ denote the connected component of $M_k\cap B^N_{\dl}(y_j)$ containing the component of $S_k^1\cap [B^N_{\dl}(y_j)\sm B^N_{\eps}(y_j)]$ converging to $\Gamma_j^+$. If $\Gamma_j^-\subset V^\ell$, let $Q^1_{j_-}$ denote the connected component of $M_k\cap B^N_{\dl}(y_j)$ containing the component of $S_k^1\cap [B^N_{\dl}(y_j)\sm B^N_{\eps}(y_j)]$ converging to $\Gamma_j^-$. Recall that if $Q^1_j\cap [B_\dl^N(y)\setminus B_\eps^N(y)]$
consists of more than one connected component, then it consists of exactly two components, one converging to  $\Gamma_j^+$ and the other to $\Gamma_j^-$. And the same is true of $Q^1_{j_-}$. If for each $y_j\in \Delta\cap \overline{V}^\ell$, $Q^1_j\cap [B_\dl^N(y)\setminus B_\eps^N(y)]$ and $Q^1_{j_-}\cap [B_\dl^N(y)\setminus B_\eps^N(y)]$ each consists of exactly one component, then $S^1_k$ would correspond to a single connected component of $M_k$ converging smoothly with multiplicity one to $\ov{V}^\ell$, and in particular $M_k$ would be disconnected. Therefore, after possibly relabelling, there exists $y_j\in \Delta\cap \overline{V}^\ell$, such that $Q^1_j\cap [B_\dl^N(y)\setminus B_\eps^N(y)]$ consists of exactly two components, one converging to  $\Gamma_j^+$ and the other to $\Gamma_j^-$. 

Notice that by the previous claim, $Q^1_j$ must be the unique such component. That is, if $\Lambda $ is another connected component of $M_k\cap B^N_{\dl}(y_j)$, then $\Lambda \cap [B_\dl^N(y)\setminus B_\eps^N(y)]$ is connected. In particular, if $\Lambda$ is the connected component of $M_k\cap B^N_{\dl}(y_j)$ containing the component of $S_k^2\cap [B^N_{\dl}(y_j)\sm B^N_{\eps}(y_j)]$ converging to $\Gamma_j^+$, then $\Lambda \cap [B_\dl^N(y)\setminus B_\eps^N(y)]$ is connected.  Note that by our choice of $S_k^1$, $\Lambda \cap [B_\dl^N(y)\setminus B_\eps^N(y)]$ must be below the component  component of $Q^1_j\cap [B_\dl^N(y)\setminus B_\eps^N(y)]$ that converges to $\Gamma_j^+$ and above the component of $Q^1_j\cap [B_\dl^N(y)\setminus B_\eps^N(y)]$ that converges to $\Gamma_j^-$. By elementary separation property, we obtain a contradiction. This  proves that $m^\ell=1$, $l=1,\dots, L$, and finishes the proof of the theorem.
\end{proof}

\section{The Bubbling Analysis}\label{sec_bubb}
 
The goal of this section is to prove the bubble-compactness theorem for $H$-hypersurfaces when $H>0$ is fixed. We shall see that, contrary to the minimal setting, the only bubbles that can occur are catenoids. We recall that the catenoid $\mathcal{C}^{n-1}\In \R^n$ is a rotationally symmetric complete minimal hypersurface with $\Ind(\mathcal{C}) = \lim_{R\to \infty} (\mathcal{C}\cap B_R(0))=1$ and finite total curvature (see e.g. \cite{TZ09} for further details). In the sequel $\mathcal{C}$ will denote any catenoid up to scaling, rotations and translations, without re-labelling.

We first recall a result of Schoen \cite[Theorem 3]{sc1} which states that for each $n\geq 3$ the only complete minimal immersions $M^{n-1}\In \R^{n}$ which are regular at infinity and have two ends are either catenoids $\mathcal{C}^{n-1}$ or a pair of hyperplanes. Combining a result of Tysk \cite[Lemma 4]{T89} with \cite[Proposition 3]{sc1} we see in particular that this implies 

\begin{lemma}\label{lem:cats}
	When $3\leq n\leq 7$ the only embedded, complete minimal hypersurfaces $M^{n-1}\In \R^n$ with Euclidean volume growth, finite index and at most two ends, are either one or two parallel planes\footnote{two parallel planes may include a single plane of multiplicity two}, or a catenoid.  
\end{lemma}

The total curvature of a hypersurface is denoted by $\mathcal{T}=\int |A|^{n-1}$ and $\mathcal{T}(\mathcal{C}^{n-1})$ denotes the total curvature of the catenoid. When $n=3$ we have $\mathcal{T}(\mathcal{C}^2)=8\pi$. 

The main result of this section is as follows. 
\begin{thm}\label{thm:bubb}
	With the same hypotheses as Theorem \ref{thm:conv2}, for each $y \in \Dl$ there exists a finite number $0<J_y\in \mathbb{N}$ of point-scale sequences (see Definition \ref{def:ps}) $\{(p^{y,\ell}_k, r^{y,\ell}_k)\}_{\ell=1}^{J_y}$ so that:
				
	\begin{enumerate}
		\item	these point-scale sequences are distinct, in the sense that for all $1\leq i\neq j\leq J_y$ 
			\[\frac{dist_g(p^{y,i}_k,p^{y,j}_k)}{r^{y,i}_k + r^{y,j}_k}\to \infty.
						\]
		Taking normal coordinates centred at $p^{y,\ell}_k$ and letting
		$\widetilde{M}^{y,\ell}_k:=M_k/r^{y,\ell}_k \subset \R^{n}$ then $\widetilde{M}^{y,\ell}_k$ converges smoothly on compact subsets to a catenoid $\mathcal{C}^{n-1}$ with multiplicity one, for all $\ell$. 


\item There exist $\dl_0,R_0>0$ so that for all $y\in \Dl$, $\dl\leq \dl_0$, $R\geq R_0$ and $k$ sufficiently large $$M_k\cap \left(B_\dl(y)\sm \cup_{\ell =1}^{J_y} B_{Rr_k^{y,\ell}}(p_k^{y,\ell})\right)$$
can be written as two smooth graphs over $T_y V=\{x^n=0\}$ with mean curvature vectors pointing in opposite directions (in suitable normal coordinates $\{x^i\}$ centred at $y$) with slope $\eta=\eta(k,R,\dl)$ satisfying
$$\lim_{\dl \to 0} \lim_{R\to \infty} \lim_{k\to \infty} \eta = 0.$$

		\item The number of catenoid bubbles $\sum_{y\in\Dl} J_y = J \leq \mathcal{I}$, and $index(V) :=\sum_{i=1}^L index(\ov{V^i})\leq \mathcal{I}-J$.
		\item There is no loss of total curvature: 
			\[
			\lim_{k\to \infty} \mathcal{T}(M_k) = \sum_{i=1}^L \mathcal{T}(\ov{V}^i) + J\mathcal{T}(\mathcal{C}^{n-1})
			\]
	where we have denoted by $\mathcal{T}(\ov{V}^i)$ and $\mathcal{T}(M_k)$ the total curvature in $(N^{n},g)$ of the hypersurfaces $\ov{V}^i$ and $M_k$, respectively. In particular, when $n=3$ we have, for all $k$ sufficiently large
		\[\chi(M_k) = \sum_{i=1}^L  \chi(\ov{V}^i) - 2J.\] 
		\item When $k$ is sufficiently large, the surfaces $M_k$ of this subsequence are pair-wise diffeomorphic to one another.
				\end{enumerate}
				
			\end{thm}

By a contradiction argument we immediately obtain the following 
\begin{corollary}\label{cor_bubb}
Given $H>0$ there exists $C=C(N,\Lambda, \mathcal{I},H)$ so that the collection of $H$-hypersurfaces with index bounded by $\mathcal{I}$ and volume bounded by $\Lambda$ has at most $C$ distinct diffeomorphism types. Furthermore for any $H$-hypersurface $M$ with the above index and volume bounds we have
\[\int_M |A|^{n-1} \leq C.\]
\end{corollary}

In order to prove Theorem \ref{thm:bubb} we will repeatedly blow-up a sequence of $H$-hypersurfaces according to a given shrinking scale centred at a sequence of points. We first introduce some terminology for this, where here and throughout this section $\dl>0$ will always denote a number satisfying $0<\dl<{\rm inj}_N$: 
\begin{definition}\label{def:ps}
Let $\{M_k\}$ be a sequence of $H$-hypersurfaces in some closed Riemannian manifold $N$. Given $x\in N$ we say that $\{(x_k,r_k)\}\In N\times \R_{>0}$ is a \textbf{point-scale sequence for} $\{M_k\}$, based at $x$, if $x_k\in M_k\cap B_\dl(x)$, $x_k\to x$ and $r_k\to 0$.

Given normal coordinates based at $B_{\dl}(x_k)$ we say that $\widetilde{M}_k\subset B^{\R^n}_{\dl/r_k}$ defined by $\widetilde{M}_k = M_k/r_k$ in these coordinates, is \textbf{a blow up at scale} $(x_k,r_k)$. 

We furthermore say that $\widetilde{M}_k$ converges \textbf{non-smoothly to a plane of multiplicity two} if there exists at least one, but finitely many points, where the convergence is smooth and graphical away from these points but not smooth and graphical across them.       
\end{definition}

With Lemma \ref{lem:cats} and this terminology we are now able to prove 
\begin{lemma}\label{lem:cats_conv}
Let $V=\cup_{\ell =1}^L \ov{V}^\ell$ be a hypersurface effectively embedded in $N$ with constant mean curvature $H>0$ and  let $\{M_k\}_{k\in \N}$ be a sequence of $H$-hypersurfaces with $\sup_k \Ind_0(M_k)<\infty$ that H-converges to $V$ with multiplicity one and let $x\in t(V)$. Let $(x_k,r_k)$ be a point-scale sequence for $\{M_k\}$ based at $x$ and $\widetilde{M}_k:=M_k/r_k \subset \R^{n}$ a blow up along this scale. Then up to subsequence and on compact subsets, $\widetilde{M}_k$ converges to a limit $\widetilde{M}_\infty$, which must pass through the origin. This happens in one of three distinct ways:  
\begin{enumerate}
\item smoothly and graphically to a catenoid
\item non-smoothly to a plane of multiplicity two
\item smoothly and graphically to a single plane or two  parallel planes.
			\end{enumerate}
In case 1 above, if $(z_k,\r_k)$ is another point-scale sequence based at $x$ with $r_k\leq \r_k$ and $$\frac{dist_g(x_k,z_k)}{r_k + \r_k}\leq C$$
then taking a blow up $\widehat{M}_k$ at scale $(z_k,\r_k)$ yields two further distinct possibilities  
\begin{enumerate}
\item[1(a)]	there exists some $K$ with $\r_k/r_k \leq K$ and  $\widehat{M}_k$ converges smoothly to a catenoid or 
\item[1(b)] $\r_k/r_k\to \infty$ and $\widehat{M}_k$ converges non-smoothly to a plane with multiplicity two. 
\end{enumerate}
 Again in either case the limit $\widehat{M}_\infty$ of the $\widehat{M}_k$'s passes through the origin.   	
\end{lemma}
\begin{proof}
Since $x\in t(V)$ then
$$\lim_{r\to 0}\fr{\|V\|(B^N_{r}(x))}{\om_{n-1} r^{n-1}} =2.$$
Now by varifold convergence, coupled with the monotonicity formula for CMC hypersurfaces (see e.g. \cite{si1}), we know that for all $\eps >0$ there exist $\eta >0$ and $r_0>0$ so that for all $z_k\in M_k\cap B^N_\eta (x)$ and $k$ sufficiently large then  
$$
\|M_k\|(B^N_{r}(z_k))\leq (2+\eps)\om_{n-1} r^{n-1}$$
for all $r\leq r_0$.

In particular if $\{(z_k,\r_k)\}$ is any point-scale sequence based at $x$ then   
\begin{equation}\label{lsmon}
\limsup_{k\to \infty}\fr{\|M_k\|(B^N_{\r_k}(z_k))}{\om_{n-1} \r_k^{n-1}} \leq 2. 	\end{equation}

Now considering $(x_k,r_k)$ and ${M}_k$ as in the statement of the lemma we will perform a blow-up at this scale in normal coordinates centred at $x_k$. Note that the metric on $N$ in these coordinates can be written $g_k = g_0 + O_k(|x|^2)$, and we may suppress the dependance on $k$ and simply write $g_k=g_0 + O(|x|^2)$ where $g_0$ denotes the Euclidean metric. We can therefore consider 
$$\widetilde{M}^1_k  \subset \Big(B^{\R^{n}}_{\delta/r_k}(0), \widetilde{g}_k\Big)$$ 
where $\widetilde{g}_k = g_0 + r_k^2 O(|x|^2)$. We have that $\widetilde{M}^1_k$ is a potentially disconnected CMC hypersurface with mean curvature $H_k=r_k H\to 0$. Moreover by \eqref{lsmon}, for any $R>0$:
\begin{equation}\label{mon}
\limsup_{k}\frac{\h^{n-1}(\widetilde{M}_k \cap B_R)}{\om_{n-1} R^{n-1}} = \limsup_k\fr{\|M_k\|(B^N_{Rr_k}(z_k))}{\om_{n-1} (Rr_k)^{n-1}}\leq 2.\end{equation}
 It follows from a standard argument using Lemma \ref{curvest} (following along the lines of e.g. \cite[Theorem 2.4, Corollary 2.5]{BS18}) that each component of $\widetilde{M}_k$ converges smoothly, away from finitely many points, to a minimal limit $M_\infty$ which has Euclidean volume growth and finite index by construction, and if the convergence is of multiplicity one  then it is smooth everywhere. $M_\infty$ has at most two ends by taking the limit as $R\to \infty$ in \eqref{mon} so by Lemma \ref{lem:cats} it must be a catenoid or at most two parallel planes. Finally, appealing again to the arguments in e.g. \cite[Theorem 2.4, Corollary 2.5]{BS18}, if the convergence is not multiplicity one (equivalently not smooth), then the limit must be (stable in compact subsets, and therefore) a plane of multiplicity two.

	For the second part of the lemma, we first note that $r_k\leq \r_k$ and $\frac{dist_g(x_k,z_k)}{r_k + \r_k}\leq C$ implies that $B_{r_k}(x_k)\In B_{2C\r_k}(z_k)$. We leave the final details to the reader as the arguments are standard, noting that in case $1(b)$ there must exist a sequence of points converging to the origin in $\widehat{M}_k$ where the second fundamental form blows up, and thus it cannot converge smoothly and graphically near the origin.     
\end{proof}

\begin{lemma}\label{lem:neck}
Let $V=\cup_{\ell =1}^L \ov{V}^\ell$ be a hypersurface effectively embedded in $N$ with constant mean curvature $H>0$ and let $\{M_k\}_{k\in \N}$ be a sequence of $H$-hypersurfaces with $\sup_k \Ind_0(M_k)<\infty$ that H-converges to $V$ with multiplicity one and let $x\in t(V)$. Suppose $(x_k,r_k)$ is a point-scale sequence for $\{M_k\}$ based at $x$ so that the blow-up at this scale converges smoothly locally to a catenoid. Suppose further that there is a positive sequence $\r_k\to 0$ with $\r_k/r_k \to \infty$ and so that $\widetilde{M}_k : = M_k/\r_k$ converges smoothly to the double plane $\{x^n=0\}$ on $B_1\sm B_\eta$ for all $\eta >0$. Then there exists $R_0<\infty$ so that for all $R\geq R_0$ $$\widetilde{M}_k\cap (B_1\sm B_{R s_k})\qquad \text{where} \qquad s_k = r_k/\r_k \to 0$$
can be written as a pair of graphs over $\{x^n=0\}$ with mean curvatures pointing in opposite directions and the graphs converge to zero in $C^1$ as first $k\to \infty$ and then $R\to \infty$.   
\end{lemma}
\begin{proof}
We will show that if $t_k\to 0$ is a sequence of positive numbers so that $s_k/t_k\to 0$, then $\widehat{M}_k=\widetilde{M}_k/t_k$ converges smoothly and graphically to $\{x^n=0\}$ on compact subsets away from the origin - in fact we need only check this in the region $B_2 \sm B_1$. Since the slope of the graph is scale-invariant, this will complete the proof. 

Lemma \ref{lem:cats_conv} tells us that (up to subsequence) $\widehat{M}_k$ converges to some plane passing through the origin. By the hypotheses of the lemma and the choice of $t_k$, this convergence happens smoothly with multiplicity two in compact subsets away from the origin. In particular there is some $(n-1)$-dimensional linear subspace $E$ of $\R^{n}$ so that $\widehat{M}_k \cap B_2\sm B_1$ can be written as two graphs over $E$ 
which are uniformly converging to zero as $k\to\infty$. We will prove below that $E=\{x^n=0\}$; this fact will be independent of the choice of sequence $t_k$ as above, and any subsequence.

Without loss of generality we will prove what we need only for the top sheet, whose mean curvature points upwards. Denote by $D_\xi$ the closed ball of radius $\xi$ centred at the origin in $\{x^n = 0\}$. Let $u_k:D_1\sm D_{1/4}\to \R$ describe the top sheet of $\widetilde{M}_k$ (whose mean curvature points upwards) and notice that $\|u_k\|_{C^l}\to 0$ for all $l$, and $H_k =\rho_k H \to 0$ is the mean curvature of $\widetilde{M}_k$.  Thus, using Proposition \ref{prop:folio} and Remark \ref{rmk:folio}, we can foliate a region of $D_{1/2}\times [-\dl,\dl]$ by CMC graphs  $v_{k}^h:D_{1/2}\to \R$ with boundary values given by $u_k + h$, $h\in\R$. Notice that as $k\to \infty$ we have that $g_k\to g_0$ and $u_k \to 0$ in $C^l$ for all $l$ which tells us that $$\|v^h_{k}-h\|_{C^{2,\al}}\to 0$$ as $k\to \infty$ which follows from Proposition \ref{prop:folio}. 

Similarly as is \cite[Lemma 3.1]{W15} (cf \cite{BS18}) we can define a diffeomorphism of this cylindrical region (via its inverse)
$$F^{-1}_k(x^1,\dots,x^{n-1}, y) = (x^1,\dots, x^{n-1}, v_k^{y-h_k}(x^1,\dots, x^{n-1}))$$
where $h_k \to 0$ is uniquely chosen so that $v_k^{-h_k}(0,\dots ,0) = 0$ (so that $F_k(0) =0$). Notice that $F_k \to Id$ as $k\to \infty$ in $C^2$, so in particular the metric $g_k$ in these coordinates is also converging to the Euclidean metric.  

We now work with these new coordinates $(x^1,\dots, x^{n-1}, y)$, on which horizontal slices $\{y = c\}$ provide a CMC foliation, and furthermore in these coordinates, the part of $\widetilde{M}_k$ described by $u_k$ takes a constant value $h_k$ at the boundary of $D_{1/2}$. Without loss of generality (by perhaps choosing a sub-sequence) we assume that $h_k \geq 0$ for all $k$ (if $h_k \leq 0$ the proof is similar). 

We now blow-up this coordinate system by a factor $1/t_k$, and let 
$$\widehat{M}_k =  \widetilde{M}_k/t_k\In D_{1/(2 t_k)}\times [-\dl/t_k,\dl/t_k].$$ Strictly speaking this is not the same $\widehat{M}_k$ as before (which was a blow-up of $\widetilde{M}_k$ in a different coordinate system) but since our two choices of coordinates are asymptotically equivalent (as $k\to \infty$), their limits are equal. In particular we still have that $\widehat{M}_k\cap B_2/B_{1}$ is uniformly graphical over $E$ (equivalently defined in either coordinates), and our goal is to prove that $E=\{y=0\} = \{x^n=0\}$. Notice that over $\partial D_{1/(2 t_k)}$, the top sheet of $\widehat{M}_k$ is described by a constant function of value $\widehat{h}_k = h_k/t_k \geq 0$, and the horizontal slices $\{y=c\}$  still provide a CMC foliation where the mean curvature of the foliation equals that of the top sheet of $\widehat{M}_k$.

For a contradiction suppose that $E\neq \{y=0\}$, which means that 
$$\min_{\widehat{M}_k\cap ((D_{1/(2 t_k)}\sm D_{1})\times \R)} y <0$$
and the minimum is not attained at a boundary point. The maximum principle for CMC graphs then implies that $\widehat{M}_k$ is globally a horizontal slice $\{y=-c_0\}$, for some $c_0<0$, which contradicts $\widehat{h}_k \geq 0$. Thus we must have $E=\{y=0\}$. 
  
Thus we have that, for $k$, $R$ sufficiently large $\widetilde{M}_k\cap (B_1\sm B_{Rs_k})$ is graphical over $\{x^n=0\}$ with slope $\eta = \eta(k,R)\to 0$ as we first send $k\to \infty$ then $R\to \infty$. 
\end{proof}

\begin{proof}[Proof of Theorem \ref{thm:bubb}]

To begin we choose $\delta$ sufficiently small so that 
\begin{equation*}
2\delta<\min\left\{\min_{\Dl\ni y_i\neq y_j\in \Dl} d_g(y_i,y_j),\,\,\, \frac{inj_N}{2}\right\}
\end{equation*}
and furthermore that $B_\dl^N(x)\cap V$ is stable for all $x\in V$. Towards the end of the proof we will consider $\dl \to 0$, but for the majority of the proof we work with some fixed $\dl$ satisfying the above. 

From now on we work with a single $y\in \Dl$ since we only need check the conclusion of the theorem for one such point chosen arbitrarily. 
\paragraph{Picking the smallest scale} 
Let
$$r^1_k = \inf \Big\{r>0 \mid M_k\cap B_r(p) \text{ is unstable for some } p\in B_\delta(y)\cap M_k \Big\}.$$
Note that with $r^1_k$ defined above, we can pick $p^1_k\in B_\delta(y)\cap M_k$ and $\delta >r^1_k>0$ such that $M_k\cap B_{3r^1_k/2}(p^1_k)$ is unstable.

We must have $p^1_k \to y$ since if not, we know that $M_k\cap B_{\ed_g(p^1_k,y)/2}(p^1_k)$ converges smoothly to $V$ and thus is eventually stable inside all such balls by the choice of $\delta$.

Furthermore $r^1_k\to 0$ as otherwise the regularity theory of Lopez-Ros and Schoen-Simon (see Lemma \ref{sreg}) would give a uniform $L^{\infty}$ estimate on the second fundamental form for $M_k\cap B_{\delta/2}(y)$ and we reach a contradiction to the fact that $y$ is a point of bad convergence.

Thus $(p^1_k,r^1_k)$ is a point scale sequence based at $y$ and we let $\widetilde{M}^1_k$ be the blow-up at this scale (see Definition \ref{def:ps}).  

The metric on $N$ in these coordinates can be written $g_k = g_0 + O_k(|x|^2)$, and we may suppress the dependance on $k$ and simply write $g_k=g_0 + O(|x|^2)$ where $g_0$ denotes the Euclidean metric. Thus we may consider 
$\widetilde{M}^1_k \subset (B^{\R^{n+1}}_{\delta/r^1_k}(0), \widetilde{g}_k)$ 
where $\widetilde{g}_k = g_0 + (r_k^1)^2 O(|x|^2)$.  
By the choice of $r^1_k$ we have that $\widetilde{M}^1_k$ is a potentially disconnected CMC hypersurface with mean curvature $H_k=r^1_k H\to 0$. 

Since $\widetilde M^1_k$ is stable inside every (Euclidean) ball of radius $\frac12$ in $(B^{\R^{n}}_{\delta/r^1_k}, \widetilde{g}_k$), by Lemma~\ref{sreg} , it converges (up to subsequence) smoothly with multiplicity one to some minimal limit $M^1_\infty$ in $\R^n$ equipped with the Euclidean metric and by Lemma \ref{lem:cats_conv} $M^1_\infty$ is either at most two planes or a catenoid.   

 $M^1_\infty$ cannot be a collection of one or two planes, as this would contradict the instability hypothesis on balls of radius $2$ centred at the origin: if $M^1_\infty$ were a collection of planes it would be strictly stable in any compact set, and this strict stability would eventually pass to $\widetilde{M}_k$ for large $k$. Thus we must have that $M^1_\infty$ is a catenoid. 
Finally, since $index(M_k\cap B_{3r^1_k/2}(p^1_k))\geq 1$,  for all large $k$ and any $\xi>0$ we have, by domain monotonicity of eigenvalues, 
$$index(M_k\setminus B_\xi(y))\leq index(M_k \setminus B_{3r^1_k/2}(p^1_k))\leq \mathcal{I}-1$$ and thus $index(V)\leq \mathcal{I}-1$. This lat step follows since there exists $\xi >0$ so that 
\begin{equation}\label{eq:ind}
	\limsup_k index (M_k\sm \cup_{y\in \Dl} B_\xi (y))\geq index(V\sm \cup_{y\in \Dl} B_\xi (y)) = index(V).
\end{equation}
Here the index of any domain is computed with respect to Dirichlet boundary conditions.

\paragraph{Picking further scales}
Now let 
$$r^2_k = \inf \Big\{r>0 \mid B_r(p)\cap( M_k\setminus B_{2r^1_k}(p^1_k)) \text{ is unstable for some } p\in B_\delta(y)\cap M_k \Big\}.$$
If $\liminf_{k\to\infty} r^2_k >0$ then the process of picking point-scale sequences stops and we go on to the neck analysis. Assuming therefore that $r^2_k\to 0$ we must also have the existence of $p^2_k\in M_k\cap B_\dl (y)$ so that $(p^2_k,r^2_k)$ is a point scale sequence based at $y$ and $(M_k\cap B_{3r^2_k/2}(p^2_k))\setminus B_{2 r^1_k}(p^1_k))$ is unstable. As before, let $\widetilde{M}^2_k$ be the blow-up at this scale which by Lemma \ref{lem:cats} converges to at most two planes or a catenoid. 

There are two distinct cases: 
\begin{enumerate}
\item 	$\frac{dist_g(p^1_k,p^2_k)}{r^1_k+r^2_k}\leq C <\infty$ (i.e. $B_{r^1_k}(p^1_k)\In B_{3C r^2_k}(p^2_k)$) and $\widetilde{M}^2_k$ converges non-smoothly to a double plane
\item $\frac{dist_g(p^1_k,p^2_k)}{r^1_k+r^2_k}\to\infty$ and $\widetilde{M}^2_k$ converges smoothly to a catenoid.  
\end{enumerate}

Indeed, in the first case we claim that the limit is attained non-smoothly and is therefore a double plane by Lemma \ref{lem:cats_conv}. For a contradiction if the limit is attained smoothly we must have $r^2_k/r^1_k \leq K$ for some $K$ and the limit is  a catenoid by Case $1(a)$ of Lemma \ref{lem:cats_conv}. However, by definition of $r^1_k$ we have $\lambda_1(M_k\cap B_{3r^1_k/2}(p^1_k))<0$ and $\lambda_1(M_k\cap B_{3r^2_k/2}(p^2_k))\setminus B_{2 r^1_k}(p^1_k))<0$. These disjoint open regions of $M_k$ remain strictly unstable for all $k$ and thus, after blowing up at scale $(p^2_k,r^2_k)$ pass to two non-empty disjoint open regions of the limiting catenoid $\Om_1, \Om_2$ for which $\lambda_1(\Om_1) \leq 0$ and $\lambda_1(\Om_2)\leq 0$. This contradicts the fact that the catenoid has index one.

In the second case we invite the reader to blow up precisely as we did for $(r^1_k,p^1_k)$ and see that $\widetilde{M}^2_k$ converges smoothly to a catenoid: at this blow up scale we once again have that, on compact subsets, $\widetilde{M}^2_k$ is stable on all balls of radius $\fr12$ and the first forming catenoid is disappearing at infinity. 

We wish to keep track of this point-scale sequence in either scenario, but in case one, the blow-up procedure produces no extra catenoid so we mark this sequence for removal later. In either case we conclude similarly as before that $index(V)\leq \mathcal{I}-2$. 

Now suppose that we have picked $j-1$ point-scale sequences $\{(r^i_k,p^i_k)\}_{i=1}^{j-1}$ satisfying 
\begin{enumerate}[a)]
\item for each $2\leq i\leq j-1$ we have $r^i_k\to 0$, $p^i_k \to y$	
\item Denoting $U_{i-1}=\cup_{s=1}^{i-1} B_{2r^s_k}(p^s_k)$ 
$$(M_k\cap B_{3r^i_k/2}(p^i_k))\setminus U_{i-1} \qquad \text{is unstable}$$ 
\item $index(M_k\sm U_{j-1}) \leq \mathcal{I}-(j-1)$ and thus $index(V)\leq \mathcal{I}-(j-1)$ by \eqref{eq:ind}

\end{enumerate}
Furthermore we suppose there are two distinct cases: 
\begin{enumerate}
\item There exists $C<\infty$ and $m<i$ so that $B_{r^m_k}(p^m_k)\In B_{C r^i_k}(p^i_k)$ and blowing up at this scale we converge non-smoothly to a double plane 
\item $\min_{m<i}\frac{dist_g(p^m_k,p^i_k)}{r^m_k+r^i_k}\to\infty$ and blowing up at this scale yields a catenoid as a smooth limit.
\end{enumerate}
We now pick the next shrinking scale (if it exists) according to 
$$r^j_k = \inf \Big\{r>0 \mid B_r(p)\cap(M_k\setminus U_{j-1})  \text{ is unstable for some } p\in B_\delta(y)\cap M_k \Big\}.$$
If $\liminf_{k\to\infty} r^j_k >0$ then the process of picking point-scale sequences stops and we go on to the neck analysis. Assuming therefore that $r^j_k\to 0$ we now perform the usual argument that first of all there exists $p^j_k\in M_k\cap B_\dl(y)$ so that 
$$(M_k\cap B_{3r^j_k/2}(p^j_k)) \setminus U_{j-1}\qquad \text{is unstable}$$ and show that once again we are in case 1. or 2. above (we leave the details to the reader) and this time $index(M_k\sm U_j)\leq \mathcal{I}-j$ implying $index(V)\leq \mathcal{I}-j$. In short, we satisfy conditions $a)-c)$ and the $j^{th}$ sequence also satisfies condition 1. or 2.

This process must stop eventually (after at most $\mathcal{I}$ iterations) and we can move on to the neck analysis, noting that if $J_y$ is the total number of distinct point-scale sequences forming at $y$ (distinct in the sense that we have removed all point-scale sequences satisfying case 1), then in particular have $index(V)\leq \mathcal{I}-J_y$ which is part 3 of the theorem.   

Before we move on let us now throw away all the marked sequences (those satisfying condition 1 above), since blowing up at these scales means that we see only a double plane passing through the origin as a weak limit, and we have finished proving part 1 of the theorem.  
\paragraph{Part 2 of the theorem} If there is only one catenoid forming at $y$ (i.e. $J_y=1$) we first pick an arbitrary $\r_k\to 0$ so that $\r_k/r^1_k \to \infty$ and we first apply Lemma \ref{lem:neck} to the blow up $\widetilde{M}_k$ at scale $(p^1_k,\r_k)$ to conclude that $\widetilde{M}_k\cap (B_{1}  \sm B_{Rr^1_k/\r_k})$ is uniformly graphical over a fixed plane $E$ (in these coordinates) with slope converging to zero as $k\to \infty$ and then $R\to \infty$. 

We now consider the point scale sequence given by $(p^1_k, \dl)$ and the corresponding blow up $\check{M}_k=M_k/\dl$. Notice that, for any $\dl >0$ we can always rotate the coordinates so that $T_y V$ is parallel to $\{x^n = 0\}$ and that for any fixed $\mu<1$, $\check{M}_k \cap B_1 \sm B_\mu$ can be written as two graphs over $\{x^n = 0\}$ with slope $\eta \to 0$ as we first send $k\to \infty$ and $\dl \to 0$. The reader can check that (by following the steps in the proof of Lemma \ref{lem:neck}) $\check{M}_k\cap B_1 \sm B_{\r_k/\dl}$ is uniformly graphical over $\{x^n = 0\}$ with slope converging to zero as $k\to \infty$ and $\dl \to 0$. Thus the orientation of the plane $\{x^n =0\}$ is passed down to the next scale (so $E=\{x^n=0\}$ above), and we recover that $\check{M}_k\cap B_1 \sm B_{Rr^1_k/\dl}$ is uniformly graphical over $\{x^n=0\}$ (equivalently over $T_y V$) with slope converging to zero as $k\to\infty$, $R\to \infty$ and finally $\dl \to 0$.     

By undoing the scaling we see that $M_k\cap (B_\dl(p^1_k) \sm B_{Rr^1_k}(p^1_k))$ is uniformly graphical over $T_y V$ with slope $\eta(k,R,\dl)$ converging to zero as $k\to \infty$, $R\to \infty$ and $\dl \to 0$.

\ 

When there is more than one bubble we simply inductively apply Lemma \ref{lem:neck} at progressively smaller scales, noting that the orientation of the limit plane (i.e. $T_y V$) is passed down to each smaller scale: the ends of the catenoids are always parallel to $T_y V$.

\paragraph{The neck analysis when $J_y >1$}   
Set $\r_k  =2\max_{j>1} {\rm dist}(p^1_k,p^j_k)$ which gives in particular that $\r_k/r^1_k \to \infty$ and Lemma \ref{lem:cats_conv} guarantees that by blowing up at scale $(p^1_k,\r_k)$ we see weak convergence of $\widetilde{M}_k=M_k/\r_k$ to a double plane. Furthermore there are $J_y$ catenoid bubbles forming inside the ball of radius $1/2$ at this scale and the convergence is smooth and graphical on compact subsets of $\R^n\sm B_1$. 

In exactly the same fashion as above we now consider $\check{M}_k=M_k/\dl$ the blow up at scale $(p^1_k,\dl)$. After rotating our coordinates so that $T_y V$ is parallel to $\{x^n=0\}$, (and again following the steps in the proof of Lemma \ref{lem:neck}) we have that $\check{M}_k \cap B_1\sm B_{\r_k/\dl}$ is uniformly graphical over $\{x^n=0\}$. 

Going back to $\widetilde{M}_k$ we now successively apply Lemma \ref{lem:neck} to each bubble forming inside $B_1$ at scale $(p^1_k,\r_k)$ to conclude part 2 of the theorem.

\paragraph{No loss of total curvature, part 4 of the theorem} 
By smooth, multiplicity one convergence away from $\Dl$ we know that 
\begin{equation}
\lim_{\dl \to 0}\lim_{k\to \infty} \int_{M_k\sm\cup_{y\in \Dl} B_\dl(y)} |A_k|^{n-1} \to \sum_i \int_{V^i} |A|^{n-1} = \int_V |A|^{n-1}.	
\end{equation}
Furthermore, by the scale invariance of the total curvature, given any point-scale sequence $(p^{\ell,y}_k,r^{\ell,y}_k)$ corresponding to a catenoid we have 
\begin{equation}
\lim_{R\to \infty} \lim_{k\to \infty}\sum_{y\in \Dl}\sum_{\ell =1}^{J_y} \int_{M_k\cap B_{Rr^{\ell,y}_k(p^{\ell,y}_k)}} |A_k|^{n-1} = J\mathcal{T}(\mathcal{C}^{n-1}).	
\end{equation}
It thus remains to check that, in each degenerating neck region between the bubble scales we have 
\begin{equation}
\lim_{\dl \to 0} \lim_{R\to \infty} \lim_{k\to \infty}  \int_{M_k\cap (\cup_{y\in \Dl} (B_\dl(y) \sm  \cup_{\ell=1}^{J_y} B_{Rr^{\ell,y}_k}(p^{\ell,y}_k))} |A_k|^{n-1} =0.	
\end{equation}
 Given that we know such regions are uniformly graphical over the limit, with slope $\eta \to 0$ in this limit, the argument now follows exactly the lines as that appearing in pp 4392 -- 4394 with the exception that equation (4.6) there must be replaced with 
$$|\Delta_{\widehat{g}_k} u_k| =  \Big| \widehat{g}_k^{\alpha\beta} \Gamma_k(\widehat{u}_k)^{n+1}_{jl} \pl{\widehat{u}_k^j}{x^\alpha}\pl{\widehat{u}_k^l}{x^\beta} + \widehat{g}_k^{\alpha\beta}(g_k)_{ij}\pl{\widehat{u}_k^j}{x^\alpha}\pl{\widehat{u}_k^l}{x^\beta} H  \Big| \leq C\eta^2 (|\widehat{u}_k| + H),$$
since we are working with CMC $H\neq 0$. This makes no difference to the remainder of the argument so we leave it to the interested reader to follow up.  

\paragraph{Finite diffeomorphism type, part 5 of the theorem} 
Notice that we have implicitly constructed a finite open cover of $\cup_k M_k$ so that in each element of the cover the $M_k$'s are pair-wise graphical over one-another, for sufficiently large $k$. Thus the $M_k$'s are globally graphical over one-another and have the same diffeomorphism type. 
\end{proof}

\subsection{Local CMC foliations}\label{folio}
Here we wish to show the existence of local CMC foliations by disks for metrics sufficiently close to the Euclidean metric, and mean curvature sufficiently small. Let $D_1\In \R^{n-1}$ be the closed unit (Euclidean) ball and $C = D_1 \times \R\subset \R^n$. For any fixed $\al \in (0,1)$ denote by $\mathcal{G}$ the collection of $C^{2,\alpha}$ Riemannian metrics on $C$ so that we can view $\mathcal{G} = C^{2,\al}(C,\mathcal{R})$ where $\mathcal{R}$ is the open set of symmetric, positive-definite $n\times n$-matrices. Let $W=C^{2,\al}(D_1)$ and $U=C_0^{2,\al}(D_1)=\{u\in W: \text{$u\equiv 0$ on $\partial D_1$}\}$.

For $(t,g,w,u)\in \R \times\mathcal{G} \times W\times U$ we denote $\mathcal{H}_g(t+w+u)$ the $g$-mean curvature of the graph $t+w+u$ with respect to the upward pointing unit normal $N_g (t+w+u)$. We consider $\Phi :\R \times\mathcal{G} \times W\times U \times C^{0,\al}(D_1)\to C^{0,\al}(D_1)$ defined by 
\begin{equation}
\Phi(t,g,w,u,H) = \mathcal{H}_g(t+w+u)-H
\end{equation}
 and notice that $\Phi$ is $C^1$ with 
 $$\Phi(t,g_E,0,0,0)=0.$$ 
 Here $g_E\in \mathcal{G}$ denotes the Euclidean metric on $C$. We now consider the derivative with respect to $u$ at $u=0$, $D_4 \Phi (t,g_E,0,0,0):C_0^{2,\al}(D_1)\to C_0^{0,\al}(D_1)$ where for $v\in C_0^{2,\al}(D_1)$ we have 
 $$D_4 \Phi (t,g_E,0,0,0)[v] = \pl{}{h}\vlinesub{h=0}\mathcal{H}_{g_E}(t+ hv).$$ 
 This is equivalent to considering an infinitesimal variation of the flat disc by the ambient vector field $V(x_1,\dots,x_n)=(0,\dots , 0, v(x_1,\dots, x_{n-1}))\in C_0^{2,\al}(C)$, whose normal component is given by $\la V,N_{g_E}(t+u_H)\ra =v$. Thus we have   
\begin{eqnarray}
D_4 \Phi (t,g_E,0,0,0)[v] = \Dl v 
\end{eqnarray}
which is a Banach space isomorphism, noting that by Schauder theory we have
 $$\|D_4 \Phi(t,g_E,0,0,0)^{-1}[f]\|_{C^{2,\al}(D_1)}\leq C\|f\|_{C^{0,\al}(D_1)}.$$

 In particular for each fixed $t$ there exists $\eps >0$ and a $C^1$ mapping 
 \begin{equation}
 \mathcal{U}:(t-\eps, t+\eps)\times B^{\mathcal{R}}_\eps (g_E) \times B^W_{\eps}(0)\times B_\eps^{C^{0,\al}}(0) \to B_{\dl}^U(0) 	
 \end{equation}
 so that whenever 
 $$(s,g,w)\in (t-\eps, t+\eps)\times B^{\mathcal{R}}_\eps (g_E) \times B^W_{\eps}(0)\times B_\eps^{C^{0,\al}}(0)$$ then $\Phi(s,g,w, \mathcal{U}(s,g,w,H),H)=0$. In particular when $g,w,H$ are fixed, $s+w+\mathcal{U}$ is a graphical foliation with mean curvatures given by the function $H$ with boundary values given by $s+w$. By uniqueness of such $H$-graphs we can carry out this local foliation for any $t$ noting that whenever two leaves have the same boundary values, they must coincide. Thus we have proven:
 \begin{proposition}\label{prop:folio}
Let $D_1\In \R^{n-1}$ be the closed unit (Euclidean) ball and $C = D_1 \times \R\subset \R^n$. Then there exists $\eps>0$ so that for any $w\in C^{2,\al}(D_1)$, $H\in C^{0,\al}(B_1)$ and Riemannian metric $g$ on $C$ satisfying 
$$\|w\|_{C^{2,\al}} + \|g-g_E\|_{C^{2,\al}} + \|H\|_{C^{0,\al}} < \eps$$
there exists a $C^{2,\al}$ foliation of graphs $u:\R\to C^{2,\al}(D_1)$ with $g$-mean curvature $H$ pointing upwards, and for each $t\in \R$, $u(t)$ has boundary values $t+w$. Furthermore $\|u\|_{C^{2,
\al}}$ depends on $t,w,g$ and $H$ in a $C^1$ way.  	
\end{proposition}
\begin{remark}\label{rmk:folio}
 If we consider $g$, $w$ and $H$ to have higher regularity we can pass this onto the foliation by the usual regularity results: in particular if $g$, is $C^{l,\al}$ for $l\geq 2$ then $\Phi_H$ is $C^{l-1}$ and we can find a $C^{l-1}$ CMC foliation, i.e. $u:\R \to C^{2,\al}$ is $C^{l-1}$ in $t$.

 \end{remark}

\center{Theodora Bourni, tbourni@utk.edu}\\
Department of Mathematics, University of Tennessee, U.S.A.

\center{Ben Sharp, B.G.Sharp@leeds.ac.uk}\\
School of Mathematics, University of Leeds, U.K.

\center{Giuseppe Tinaglia, giuseppe.tinaglia@kcl.ac.uk}\\
Department of Mathematics, King's College London, U.K.

\bibliographystyle{plain}
\bibliography{bill.bib}
\end{document}